\documentclass{amsart}
\makeatletter
\def\paragraph{\@startsection{paragraph}{4}
	\z@\z@{-\fontdimen2\font}%
	{\normalfont\bfseries}}
\def\subsubsection{\@startsection{subsubsection}{3}
	\z@{.5\linespacing\@plus.7\linespacing}{-.5em}%
	{\normalfont\bfseries}}
\makeatother

\usepackage{hyperref}
\usepackage{amsmath,amssymb,amsfonts,amsthm}
\usepackage{graphics,graphicx}
\usepackage{subcaption}
\usepackage[shortlabels]{enumitem}
\usepackage{color}
\usepackage{threeparttable}
\usepackage{tabularx}

\newtheorem{thm}{Theorem}[section]
\newtheorem{lem}[thm]{Lemma}
\newtheorem{prop}[thm]{Proposition}
\newtheorem{cor}[thm]{Corollary}

\theoremstyle{definition}
\newtheorem{eg}[thm]{Example}
\newtheorem{assump}[thm]{Assumption}

\theoremstyle{remark}
\newtheorem{rmk}[thm]{Remark}

\numberwithin{equation}{section}

\newcommand\R{\mathbb R}
\newcommand\verB[1]{#1}

\begin{document}

\title[Analysis of Decoder Width for Parametric PDEs]{Analysis of the Decoder Width for Parametric Partial Differential Equations}


\author{Zhanhong Ye}
\address{Beijing International Center for Mathematical Research, Peking University, Beijing, China}
\email{yezhanhong@pku.edu.cn}

\author{Hongsheng Liu}
\address{Central Software Institute, Huawei Technologies Co. Ltd, Hangzhou, China}
\email{liuhongsheng4@huawei.com}

\author{Zidong Wang}
\address{Central Software Institute, Huawei Technologies Co. Ltd, Hangzhou, China}
\email{wang1@huawei.com}

\author{Bin Dong}
\address{Beijing International Center for Mathematical Research, Peking University, Beijing, China}
\address{Center for Machine Learning Research, Peking University, Beijing, China}
\email{dongbin@math.pku.edu.cn}
\thanks{Bin Dong is the corresponding author. }


\date{}

\dedicatory{}

\begin{abstract}
	Recently, Meta-Auto-Decoder (MAD) was proposed as a novel reduced order model (ROM) for solving parametric partial differential equations (PDEs), 
	and the best possible performance of this method can be quantified by the decoder width.  
	This paper aims to provide a theoretical analysis related to the decoder width. 
	The solution sets of several parametric PDEs are examined, and
	the upper bounds of the corresponding decoder widths are estimated. 
	In addition to the elliptic and the parabolic equations on a fixed domain, we investigate the advection equations that present challenges for classical linear ROMs, 
	as well as the elliptic equations with the computational domain shape as a variable PDE parameter. 
	The resulting fast decay rates of the decoder widths indicate the promising potential of MAD in addressing these problems. 
\end{abstract}

\maketitle

\section{Introduction}\label{sec:intro}
Many important phenomena in science and engineering can be modeled by the so-called parametric partial differential equations (PDEs), which can be formulated as
\begin{equation}\label{eq:def_pPDE}
	\mathcal{L}_{\widetilde{\pmb{x}}}^{\gamma_1} u = 0, \ {\widetilde{\pmb{x}}} \in \Omega \subset \R^d,\qquad
	\mathcal{B}_{\widetilde{\pmb{x}}}^{\gamma_2} u = 0, \ {\widetilde{\pmb{x}}} \in \partial \Omega
.\end{equation}
Here, $\mathcal{L}^{\gamma_1}$ and $\mathcal{B}^{\gamma_2}$ are differential operators parametrized by $\gamma_1$ and $\gamma_2$, respectively, and ${\widetilde{\pmb{x}}}$ denotes the independent variable in spatiotemporal-dependent PDEs. 
Let $\mathcal{A}$ be the space of parameters. 
For a given set of parameters of the PDEs $\eta=(\gamma_1,\gamma_2,\Omega) \in \mathcal{A}$, we assume~\eqref{eq:def_pPDE} to admit a unique solution $u^\eta\in\mathcal{U}(\Omega;\R^{d_u})$, 
where $\mathcal{U}(\Omega;\R^{d_u})$ is a Banach space of functions defined on $\Omega$. 
In the simpler case when the computational domain $\Omega$ is fixed, we shall denote for simplicity $\mathcal{U}=\mathcal{U}(\Omega;\R^{d_u})$. 
Taking $S:\eta\mapsto u^\eta$ to be the solution mapping, the set of solutions is then given by $S(\mathcal{A})=\{u^\eta\mid\eta\in\mathcal{A}\}$. 

Solving the parametric PDEs~\eqref{eq:def_pPDE} quickly for a series of $\eta$'s 
is crucial in many application domains, such as inverse problems~\cite{Fan2020SolvingEI}, inverse design~\cite{Lye2021ISMO}, optimal control~\cite{Chen2023BenchmarkingDO} and uncertainty quantification~\cite{Thanh2012ExtremeSU,Khoo2021SolvingPP}.
In order to obtain an efficient solver for the parametric PDEs, the reduced order modeling technique aims to construct a suitable low-dimensional trial manifold containing (at least approximately) the solution set $S(\mathcal{A})$. 
By restricting the original system~\eqref{eq:def_pPDE} or its discretized version to this trial manifold, a reduced order model (ROM) is produced, whose low-dimensional nature would enable faster solving. 
However, taking into account the intrinsic complexity of the solution set, the solution provided by the ROM may not achieve the desired accuracy, especially when the dimensionality of the trial manifold is insufficient. 
The concept of \emph{widths} 
is then introduced to quantify the best possible accuracy using an $n$-dimensional trial manifold. 
Different types of trial manifolds are considered in different ROMs, and the definitions of widths would vary accordingly. 

A widespread family of ROMs, such as reduced basis methods (RBMs)~\cite{Quarteroni2015ReducedBM}, utilize only those trial manifolds that are linear subspaces, 
and represent the approximate solution as a linear combination of basis functions 
$\hat u=\sum_{i=1}^{n}\lambda_iu_i.$
The Kolmogorov width~\eqref{eq:KolWidth} serves as a theoretical tool for analyzing such ROMs. 
As the solution set of advection-dominated parametric PDEs typically exhibits a slow decaying Kolmogorov width~\cite{Ohlberger16ReducedBM,Greif2019DecayKN}, 
these linear ROMs would struggle to balance between efficiency and accuracy when solving these equations, and this has become a long-standing challenge in the field of parametric PDEs. 

Making use of the nonlinear representation of the neural networks (NNs), 
many novel ROMs are proposed~\cite{Lee2020ModelRD,Fresca2021AComprehensiveDL,Fresca2022DeepLB,Gruber22AComparisonNN,Chen2022CROMCR,Ye2023MetaAD}. 
Specifically, the recently proposed ROM named Meta-Auto-Decoder (MAD)~\cite{Huang2022MetaAD,Ye2023MetaAD} seeks to construct a Lipschitz continuous mapping $D:Z\to\mathcal{U}$, where $Z=\R^n$ represents the latent space. 
Denoting $Z_B=\{\pmb{z}\mid\|\pmb{z}\|_2\le 1\}$ to be the closed unit ball of $Z$, MAD takes $D(Z_B)$ as the trial manifold, and solves~\eqref{eq:def_pPDE} by searching on this trial manifold (MAD-L) or possibly in a neighborhood of it (MAD-LM). 
The corresponding width associated with this novel ROM is termed the \emph{decoder width}, which will be further introduced in Section~\ref{sec:decWidth}. 
In contrast to the Kolmogorov width, the decoder width exhibits a much faster decay for certain advection equations (Section~\ref{sec:advVC}),
indicating the superiority of MAD over classical linear ROMs. 

Moreover, MAD is mesh-free since it
represents the decoder mapping $D:Z\to\mathcal{U}$ as $z\mapsto u_\theta(\cdot,z)$, where $u_\theta$ can be modeled \verB{(in the simplest case)} as a neural network that takes the concatenation of $\widetilde{\pmb{x}}$ and $\pmb{z}$ as input. 
On the contrary, typical ROMs are based on classical numerical methods such as finite difference method (FDM) and finite element method (FEM), in which the computational domain $\Omega$ is often discretized, 
and we cannot evaluate the solution $u^\eta$ to~\eqref{eq:def_pPDE} at arbitrary spatiotemporal points $\widetilde{\pmb{x}}$ without interpolation. 
In the case of
parametric PDEs on a variable domain, i.e. those including $\Omega$ as part of the variable parameter $\eta$, further convenience would be brought by the mesh-free nature. 
Typical discretization-based ROMs need to assume the existence of a \emph{reference domain} $\Omega^\text{ref}$, and associate to each possible $\Omega$ with an explicit domain transformation $T_\Omega:\Omega^\text{ref}\to\Omega$, 
which could be hard to design by hand for complex-shaped domains, and has the additional prerequisite that no topology changes in $\Omega$ should occur. 
These are no longer necessary for the MAD method, as any $\widetilde{\pmb{x}}\in\R^d$ can be taken as the input of a neural network. 
The corresponding notion of decoder width can be introduced by using a \emph{master domain} $\Omega^\text{m}$, as will be further explained in Section~\ref{sec:decWidthVD}. 
To the best of our knowledge, this is the first notion of width that can deal with domain deformation without using a reference domain. 


The current manuscript primarily focuses on analyzing the decoder width,
which quantifies the best possible performance of the MAD method. 
In comparison to the stable manifold width~\eqref{eq:SManiWidth}, the decoder width does not exceed its value (Proposition~\ref{prop:decoVSsmani}), and the difference is sometimes real (Example~\ref{eg:decoVSsmani1}). 
We also estimate the decoder widths of the solution sets of several common parametric PDEs. 
For the elliptic and the parabolic equations, including those on variable domains, the decoder widths would become zero for finite values of $n$ if $\mathcal{A}$ is finite-dimensional (Example~\ref{eg:ellipEg1}, \ref{eg:parabEg1}, \ref{eg:eVDmovingDisk} and~\ref{eg:eVDdefmHole}). 
Even when $\mathcal{A}$ is infinite-dimensional, the decoder width can still exhibit a fast decay rate, depending on the internal complexity of $\mathcal{A}$ (Example~\ref{eg:parabEg2} and~\ref{eg:eVDctvc}). 
In the case of the advection equations, an exponential decay rate is guaranteed when using the $L^2$-norm (Example~\ref{eg:advEg1}), 
and the decoder width may become zero if the $L^1$-norm is employed instead (Example~\ref{eg:advEg1_L1}). 
We also compare the decay rates of the decoder width with that of the Kolmogorov width and the manifold width in Table~\ref{tab:compWidthEqn}. 
These results demonstrate the promising potential of the MAD method in solving these parametric PDEs, 
and motivate the choice of the latent dimension $n$ in its implementation. 

The rest of the paper is orgnized as follows. 
In Section~\ref{sec:relWork} we provide a brief review of the related work, including the definitions of existing widths. 
In Section~\ref{sec:decWidth}, 
we present the definition of the decoder width along with its fundamental properties, covering both the fixed domain case and the variable domain case. 
In Section~\ref{sec:decWidthEst}, we examine the solution sets of several parametric PDEs, and estimate the upper bounds of the corresponding decoder widths. 
We shall only state the conclusions in Section~\ref{sec:decWidth} and \ref{sec:decWidthEst}, and defer all the proofs till Section~\ref{sec:proofs}. 
Our conclusions as well as further discussions will be given in Section~\ref{sec:conclusion}. 

\begin{table}[tb]
\hspace*{-\textwidth}
\begin{threeparttable}
	\centering
	\caption{A comparison of the decay rate of different widths, taking $\mathcal{K}$ as the solution sets of several parametric PDEs (on a fixed domain $\Omega$). We assume $\mathcal{A}$ has finite-dimensionality $\mathsf{d}^\mathcal{A}=\dim\mathcal{A}<\infty$. A ``--'' means the result is currently unknown to the authors. }
	\label{tab:compWidthEqn}
	\begin{tabular}{cccc}
	\hline
	PDE type & Kolmogorov width $d_n$ & manifold width $d_n^\text{Mani}$ & decoder width $d_{n,l}^\text{Deco}$\\
	\hline
	elliptic & $O(n\exp(-cn^{1/\mathsf{d}^\mathcal{A}}))$~\cite{Tran17AnalysisQO} & 0 as $n\ge \mathsf{d}^\mathcal{A}$~\cite{Franco2021DeepLA} & 0 as $n\ge \mathsf{d}^\mathcal{A}$ (Eg.\ref{eg:ellipEg1})\\
	parabolic & -- & -- & 0 as $n\ge \mathsf{d}^\mathcal{A}$ (Eg.\ref{eg:parabEg1})\\
	advection ($L^1$) & $O(1/n)$~\cite{Petrova2022LipschitzW} & -- & 0 as $n\ge \mathsf{d}^\mathcal{A}$ (Eg.\ref{eg:advEg1_L1})\\
	advection ($L^2$) & $O(1/\sqrt n)$~\cite{Ohlberger16ReducedBM} & $O(e^{-cn})$\tnote{*} &  $O(e^{-cn})$ (Eg.\ref{eg:advEg1})\\
	\hline
	\end{tabular}
	\begin{tablenotes}
		\item[*]{The proof is the same as Example~\ref{eg:advEg1} if we apply~\ref{eq:entSMani} instead of Corollary~\ref{cor:decoVSentNum}. }
	\end{tablenotes}
\end{threeparttable}
\hspace*{-\textwidth}
\end{table}

\section{Related Work}\label{sec:relWork}
\subsection{Widths and Parametric PDEs}
Let $\mathcal{U}$ be a Banach space, and $\mathcal{K}\subset\mathcal{U}$ be a compact subset. 
If we aim at approximating $\mathcal{K}$ using a linear trial manifold as in classical ROMs, 
then the best possible accuracy can be measured by the well-known Kolmogorov width~\cite{Kolmogoroff1936UberDB}. 
To be more specific, given a linear subspace $U_n\subset\mathcal{U}$ of dimension $n$, its performance in approximating the elements of $\mathcal{K}$ is evaluated by the worst case error
\begin{equation*}
\sup_{u\in\mathcal{K}}d_{\mathcal{U}}(u,U_n) = \sup_{u\in\mathcal{K}}\inf_{v\in U_n}\|u-v\|_{\mathcal{U}} .
\end{equation*}
The \emph{Kolmogorov width} is then defined by taking infimum over all such linear subspaces
\begin{equation}\label{eq:KolWidth}
	d_n(\mathcal{K})_\mathcal{U}=\inf_{U_n}\sup_{u\in\mathcal{K}}\inf_{v\in U_n}\|u-v\|_{\mathcal{U}} 
.\end{equation}
We may write $d_n(\mathcal{K})=d_n(\mathcal{K})_\mathcal{U}$ for simplicity when $\mathcal{U}$ is clear from the context, and this applies to all notions of widths to be introduced below.
In the context of parametric PDEs, the case of particular interest is taking $\mathcal{K}=S(\mathcal{A})$ to be the solution set. 
For the elliptic equations, the corresponding Kolmogorov width of the solution set exhibits rapid decay~\cite{Cohen2010AnalyticRA,Tran17AnalysisQO}, indicating that conventional linear ROMs can achieve a satisfactory accuracy with a reasonable latent dimension $n$. 
However, the case is quite different for advection-dominant problems~\cite{Ohlberger16ReducedBM,Greif2019DecayKN}, 
posing significant performance challenges for linear ROMs. 

In light of the limitations of the linear ansatz, nonlinear alternatives to the Kolmogorov width are considered, and the manifold width proposed in~\cite{DeVore1989OptimalNA} is a noticeable candidate. 
Let $Z=\R^n$ be a fixed $n$-dimensional latent space, and $E:\mathcal{K}\to Z$, $D:Z\to\mathcal{U}$ be two continuous mappings, we consider the worst case reconstruction error given by
\begin{equation*}
\sup_{u\in\mathcal{K}}\|u-D(E(u))\|_{\mathcal{U}}=\sup_{u\in\mathcal{K}}\|u-D(E(u))\|_{\mathcal{U}}
,\end{equation*}
and the \emph{manifold width} is defined by taking infimum over all possible mapping pairs $(E,D)$:
\begin{equation}\label{eq:ManiWidth}
d_n^{\mathrm{Mani}}(\mathcal{K})=\inf_{E,D}\sup_{u\in\mathcal{K}}\|u-D(E(u))\|_{\mathcal{U}} .
\end{equation}
For the solution set of certain elliptic equations, this width is known to become zero even for finite $n$~\cite{Franco2021DeepLA}. 
If we further require these two mappings to be $l$-Lipschitz continuous
, a variant named the \emph{stable manifold width}~\cite{Cohen2022OptimalSN} is then obtained as%
\footnote{The original definition given in~\cite{Cohen2022OptimalSN} actually takes infimum over all possible norms $\|\cdot\|_Z$ on $Z$ along with $E,D$. 
Here we choose a fixed norm $\|\cdot\|_Z=\|\cdot\|_2$ as in the decoder width~\eqref{eq:decWidth}, which won't make a difference for our purpose, and~\eqref{eq:entSMani} is still valid. }
\begin{equation}\label{eq:SManiWidth}
d_{n,l}^{\mathrm{SMani}}(\mathcal{K})=\inf_{E,D\ l\text{-Lip}}\sup_{u\in\mathcal{K}}\|u-D(E(u))\|_{\mathcal{U}} .
\end{equation}
We note that some recently proposed ROMs~\cite{Lee2020ModelRD,Fresca2021AComprehensiveDL,Fresca2022DeepLB,Gruber22AComparisonNN,Chen2022CROMCR} find the trial manifold by training a convolutional autoencoder, and these notions of widths can serve as a quantified criterion for their best possible performance. 

Another notion termed the entropy number, though not named as a ``width'', is also noteworthy due to its close relationship to these concepts. 
Given an integer $n>0$, the \emph{entropy number} is defined to be the infimum of all $\epsilon>0$ such that $2^n$ balls of radius $\epsilon$ cover $\mathcal{K}$~\cite{Carl1981EntropyNS}, or equivalently
\begin{equation}\label{eq:ent_num_def}
	\epsilon_n(\mathcal{K})=\inf_{V_n}\sup_{u\in \mathcal{K}}\inf_{v\in V_n}\|u-v\|_\mathcal{U}
,\end{equation}
where the first infimum is taken over all subsets $V_n\subset\mathcal{U}$ containing no more than $2^n$ points. 
If the additional restriction $V_n\subset\mathcal{K}$ is imposed, we obtain the so-called \emph{inner entropy number} $\tilde\epsilon_n(\mathcal{K})$, 
and $\epsilon_n(\mathcal{K})\le\tilde\epsilon_n(\mathcal{K})\le 2\epsilon_n(\mathcal{K})$ holds. 
When $\mathcal{U}$ is a Hilbert space, \cite[Theorem~4.1]{Cohen2022OptimalSN} asserts the inequality
\begin{equation}\label{eq:entSMani}
d^\text{SMani}_{26n,2}(\mathcal{K})\le 3\epsilon_n(\mathcal{K})
.\end{equation}

More variants of widths exist in the realm of metric geometry, including the Alexandroff width~\cite{Alexandroff1933berDU} and the Urysohn width~\cite{Gromov88WidthRI}. 
The interested readers may refer to~\cite{DeVore93WaveletCN} for a further discussion on this topic. 

\subsection{PDEs on a Variable Domain}
Some parametric PDEs may take the shape of the computational domain as the variable parameter, i.e. $\eta=\Omega$, 
which have attracted attention due to both theoretical and engineering interests. 
Solving shape optimization problems typically requires to apply sensitivity analysis to the solutions to~\eqref{eq:def_pPDE}. 
To be more specific, we consider a class of domains $\Omega_\epsilon$ parametrized by a small parameter $|\epsilon|\ll 1$, and the derivative of $u^{\Omega_\epsilon}$ (after transformed to a reference domain) with respect to $\epsilon$ is analyzed~\cite{Haug1988DesignSA,Sokolowski1992IntroductionSO,Marius2003LinearEB}. 
Another bunch of work investigates scenarios where the domain is perturbed singularly. 
Given a series of domains $\Omega_n$ approaching $\Omega$, the convergence of the corresponding solutions $u^{\Omega_n}$ is studied~\cite{Dancer1997DomainPE,Daners2008DomainPL,Druet2013StabilityEP}, and the solutions on different domains are compared via direct restrictions to subdomains. 
However, these works primarily focus on the limiting behavior of a varying domain.
In contrast, the concept of widths is related to the whole solution set $S(\mathcal{A})=S(\{\Omega\})$, which takes into account the entire distribution of all possible domains. 
The main challenge in analyzing widths in this context arises from the fact that the solutions involved are defined on different domains and, consequently, belong to different function spaces. 
Transforming all these solutions to a fixed reference domain appears to be one resolution to this issue~\cite{Quarteroni2015ReducedBM}, but lacks flexibility as we have discussed in Section~\ref{sec:intro}. 
To the best of our knowledge, we are the first to introduce a version of width for variable-domain PDEs without involving a reference domain (Section~\ref{sec:decWidthVD}), 
and provide a theoretical estimation of it for several special cases (Section~\ref{sec:ellpVD}).

\section{Definition and Basic Properties of the Decoder Width}\label{sec:decWidth}
For the rest of this paper, we denote $\|\cdot\|_{L^p}$ to be the $L^p$-norm of a function, and $\|\cdot\|_{H_0^1}$ similarly. 
The $\ell^p$-norm of a vector or matrix is denoted as $\|\cdot\|_p$, making
$\ell^p(\R^d)=(\R^d,\|\cdot\|_p)$ to be a finite-dimensional Banach space, as well as $\ell^p(\R^{d\times d})=(\R^{d\times d},\|\cdot\|_p)$ endowed with the matrix norm. 
If $\mathcal{V}$ is a Banach space, 
\begin{equation}\label{eq:defBR}
	\bar B_R(\mathcal{V})=\{v\in\mathcal{V}\mid\|v\|_{\mathcal{V}}\le R\}
\end{equation}
represents the closed ball of radius $R>0$ in $\mathcal{V}$. 

\subsection{The Decoder Width on a Fixed Domain}\label{sec:decWidthFD}
We recall the notion of the decoder width initially introduced in the context of the MAD method~\cite{Ye2023MetaAD}. 
Given a fixed domain $\Omega$, $\mathcal{U}=\mathcal{U}(\Omega;\R^{d_u})$ is a fixed Banach space, and we are interested in the approximation of a compact subset $\mathcal{K}\subset\mathcal{U}$. 
MAD constructs the trial manifold by using a decoder mapping $D:Z\to\mathcal{U}$ similar to the autoencoders, but does not involve the encoder mapping $E:\mathcal{K}\to Z$. 
The corresponding \emph{decoder width} would thereby be different from the stable manifold width~\eqref{eq:SManiWidth}, and is defined as
\begin{equation}\label{eq:decWidth}\begin{split}
	d_{n,l}^\mathrm{Deco}(\mathcal{K})&=\inf_{D\ l\text{-Lip}}\sup_{u\in\mathcal{K}}d_{\mathcal{U}}(u,\{D(\pmb{z})\mid\|\pmb{z}\|_2\le 1\})
	\\&=\inf_{D\ l\text{-Lip}}\sup_{u\in\mathcal{K}}\inf_{\pmb{z}\in Z_B}\|u-D(\pmb{z})\|_\mathcal{U}
,\end{split}\end{equation}
where the first infimum is taken over all mappings $D:Z\to\mathcal{U}$ that are $l$-Lipschitz continuous. 
The constraint $\|\pmb{z}\|_2\le 1$ and the Lipschitz continuity condition are required to avoid highly irregular mappings resembling the space-filling curves, which will be further discussed in Section~\ref{sec:conclusion}. 

For the case of a fixed domain, the decoder width is very similar to the Lipschitz width proposed in~\cite{Petrova2022LipschitzW}. 
In the definition of the Lipschitz width, the latent space $Z=\R^n$ is assumed to be a general finite-dimensional Banach space, and an additional infimum over all possible norms $\|\cdot\|_Z$ on $Z$ is taken before the first infimum in~\eqref{eq:decWidth}, 
that is, 
\[d_{n,l}^\mathrm{Lip}(\mathcal{K})=\inf_{\|\cdot\|_Z}\inf_{D\ l\text{-Lip}}\sup_{u\in\mathcal{K}}\inf_{\pmb{z}\in Z_B}\|u-D(\pmb{z})\|_\mathcal{U} .\]
However, the decoder width sticks to the $\ell^2$-norm on $Z$. 
As the $\ell^2$-norm is typically more convenient for numerical optimization, we take the decoder width as the more suitable notion of width related to the MAD method. 
Indeed, it is worth noting that
the Lipschitz width is not greater than the decoder width, and any upper bound derived for the decoder width automatically serves as an upper bound for the Lipschitz width as well. 

In order to give an analysis of the decoder width, we find it useful to introduce the concept of the \emph{restricted decoder width} $d_{n,l}^\text{rDeco}(\mathcal{K})$, which shares the same expression as~\eqref{eq:decWidth} but with the first infimum taken over all $l$-Lipschitz mappings $D:Z\to\mathcal{K}$. 
In other words, it imposes the additional requirement that the image of $D$ should lie inside $\mathcal{K}$, and we have the obvious inequality $d_{n,l}^\text{Deco}(\mathcal{K})\le d_{n,l}^\text{rDeco}(\mathcal{K})$. 

The following properties will be useful for the analysis of the decoder width: 
\begin{prop}\label{prop:decWidthLip}
	Assume $\mathcal{K}=\mathcal{F(K_V)}$, where
	$\mathcal{K}_\mathcal{V}\subset\mathcal{V}$ is a compact subset,
	$\mathcal{V}$ is another Banach space,
	and $\mathcal{F:K_V\to U}$ is an $L_\mathcal{F}$-Lipschitz mapping.
	Then we have
	$d_{n,L_\mathcal{F}l}^\text{rDeco}(\mathcal{K})\le L_\mathcal{F}d_{n,l}^\text{rDeco}(\mathcal{K}_\mathcal{V}) .$
	The obvious corollary is
	$d_{n,L_\mathcal{F}l}^\text{Deco}(\mathcal{K})\le L_\mathcal{F}d_{n,l}^\text{rDeco}(\mathcal{K}_\mathcal{V}) .$
\end{prop}
\begin{prop}\label{prop:decoVSsmani}
	Assume the compact subset $\mathcal{K}$ has diameter $\delta_{\mathcal{K}}=\mathrm{diam}(\mathcal{K})=\sup_{u,v\in\mathcal{K}}\|u-v\|_{\mathcal{U}}<+\infty$. Then
	$
	d_{n,\delta_{\mathcal{K}} l^2}^\mathrm{Deco}(\mathcal{K})\le d_{n,l}^\mathrm{SMani}(\mathcal{K}) .
	$
\end{prop}
\begin{cor}\label{cor:decoVSentNum}
	If $\mathcal{U}$ is a Hilbert space, and $\mathcal{K}$ has diameter $\delta_{\mathcal{K}}=\mathrm{diam}(\mathcal{K})<+\infty$, then we shall have
	$d^\text{Deco}_{26n,4\delta_{\mathcal{K}}}(\mathcal{K})\le 3\epsilon_n(\mathcal{K}) .$
\end{cor}
Compared with the stable manifold width~\eqref{eq:SManiWidth}, the decoder width would not take a greater value according to Proposition~\ref{prop:decoVSsmani}. 
Here a strict inequality can appear, as is shown in the following simple example: 
\begin{eg}\label{eg:decoVSsmani1}
	Let $\mathcal{U}=\ell^2(\R^2)$ be the Euclidean plane, and $\mathcal{K}=\{u\in\mathcal{U}\mid\|u\|_2=1\}=S^1$ be the unit circle. 
	Then we have $d_{1,l}^\text{SMani}(\mathcal{K})\ge 1$ for any $l>0$, but $d_{1,\pi}^\text{Deco}(\mathcal{K})=0$. 
	See Figure~\ref{fig:deco_vs_mani} for an illustration. 
\end{eg}
\begin{figure}[tb]
	\centering
	\includegraphics[width=0.35\linewidth]{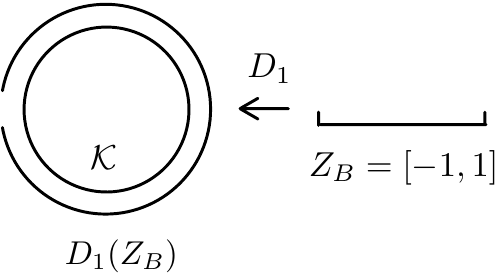}
	\caption{An illustration of Example~\ref{eg:decoVSsmani1}. With an appropriately chosen decoder mapping $D_1:Z\to\mathcal{U}$, $\mathcal{K}$ can be fitted perfectly by the trial manifold $D_1(Z_B)$, and the decoder width is therefore zero. }%
	\label{fig:deco_vs_mani}
\end{figure}
\begin{rmk}
	The Lipschitz width exhibits similar properties to the decoder width. 
	The inequality presented in Proposition~\ref{prop:decoVSsmani} also appears in~\cite[Theorem 5.1]{Petrova2022LipschitzW}, 
	and~\cite[Theorem 5.3]{Petrova2022LipschitzW} generalizes Example~\ref{eg:decoVSsmani1} to the case of all $n\ge 1$. 
	However, for the Lipschitz width, an inequality stronger than Corollary~\ref{cor:decoVSentNum} can be given (see~\cite[Theorem 3.3]{Petrova2022LipschitzW}), and the proof relies on choosing a norm different than $\ell^2$ (more specifically, $\ell^\infty$) on $Z$. 
\end{rmk}
Corollary~\ref{cor:decoVSentNum} can be derived directly by combining Proposition~\ref{prop:decoVSsmani} with~\eqref{eq:entSMani}. 
The proof of the other assertions above will be given in Section~\ref{sec:proofs}. 

\subsection{The Decoder Width on a Variable Domain}\label{sec:decWidthVD}
Another more complicated case of interest would involve a variable domain, rather than a fixed one. 
Now let $\mathcal{K}$ be a set of functions, and each element $u\in\mathcal{K}$ is associated with a domain $\Omega_u\subset\R^d$ such that $u\in\mathcal{U}(\Omega_u;\R^{d_u})$\footnote{Here $d_u>0$ is a fixed integer independent of $u$. }. 
This scenario arises when the shape of the computational domain $\Omega$ is part of the variable parameter in the parametric PDEs. 
In this case, the unique solution to~\eqref{eq:def_pPDE} for any PDE parameter $\eta=(\gamma_1,\gamma_2,\Omega)\in\mathcal{A}$ is denoted as $u^\eta$, and we have $\Omega_{u^\eta}=\Omega$. 
This case poses a challenge, as different elements of $\mathcal{K}$ may lie in different function spaces, and the distance between them is not readily defined. 
The natural resolution to this issue is to modify all these functions to fit into the same function space, which can be achived by using either a reference domain or a master domain. 

In the first case, we assume there is a \emph{reference domain} $\Omega^\text{ref}$, and every function $u\in\mathcal{K}$ is associated with a diffeomorphism $T_u:\Omega^\text{ref}\to\Omega_u$, which satisfies certain regularity conditions (depending on $\mathcal{U}$). 
Making use of this diffeomorphism, we may obtain another function on the reference domain
\begin{equation*}
\bar{u}(\widetilde{\pmb{x}})=u(T_u(\widetilde{\pmb{x}}))\in\mathcal{U}(\Omega^\text{ref};\R^{d_u}) .
\end{equation*}
The modified set of functions $\bar{\mathcal{K}}=\{\bar{u}(\widetilde{\pmb{x}})\mid u\in\mathcal{K}\}$ is now a subset of the function space $\mathcal{U}(\Omega^\text{ref};\R^{d_u})$, 
and all previous definitions of widths (including the decoder width) can be applied. 
This reformulation is convenient for classical mesh-based ROMs~\cite{Quarteroni2015ReducedBM}, but may lack flexibility as discussed in Section~\ref{sec:intro}. 

In the second case, we assume instead there is a \emph{master domain} $\Omega^\text{m}$ containing all possible $\Omega$'s as its subdomain, i.e., $\Omega_u\subseteq\Omega^\text{m}$ for any $u\in\mathcal{K}$. 
Using restrictions instead of domain transformations, 
the decoder width is now defined as
\begin{equation}\label{eq:decWidthVD}
	d_{n,l}^\mathrm{Deco}(\mathcal{K})=\inf_{D\ l\text{-Lip}}\sup_{u\in\mathcal{K}}\inf_{\pmb{z}\in Z_B}
	\Bigl\|u-D(\pmb{z})\vert_{\Omega_u}\Bigr\|_{\mathcal{U}(\Omega_u;\R^{d_u})}
,\end{equation}
where the first infimum is taken over all mappings $D:Z\to\mathcal{U}(\Omega^\text{m};\R^{d_u})$ that are $l$-Lipschitz continuous. 
Conceptually, it aims to approximate the target function $u$ by restricting the decoded function $D(\pmb{z})$ 
to the subdomain $\Omega_u\subseteq\Omega^\text{m}$ on which $u$ is defined. 
This alternative reformulation is prefered in the context of the MAD method, since we take neural networks as the mesh-free ansatz of PDE solutions, which can be evaluated for all $\widetilde{\pmb{x}}\in\R^d$. 

In order to estimate this type of decoder width of the solution set $S(\mathcal{A})$, we may seek to construct an extended solution mapping $\bar{S}:\mathcal{A\to U}(\Omega^\text{m};\R^{d_u}),\eta\mapsto\bar u^\eta$
which satisfies $\bar u^\eta|_\Omega=u^\eta$ for any $\eta=(\gamma_1,\gamma_2,\Omega)\in\mathcal{A}$. 
Let $\bar{\mathcal{K}}=\{\bar u^\eta\mid\eta\in\mathcal{A}\}=\bar{S}(\mathcal{A})$,
it is then easy to see that
$d_{n,l}^\mathrm{Deco}(\mathcal{K})\le d_{n,l}^\mathrm{Deco}(\bar{\mathcal{K}}) $
holds, since we have
\[ 
	\Bigl\|u^\eta-D(\pmb{z})|_{\Omega}\Bigr\|_{\mathcal{U}(\Omega;\R^{d_u})}
	=\Bigl\|(\bar u^\eta-D(\pmb{z}))|_{\Omega}\Bigr\|_{\mathcal{U}(\Omega;\R^{d_u})}
	\le\|\bar u^\eta-D(\pmb{z})\|_{\mathcal{U}(\Omega^\text{m};\R^{d_u})}
.\]
In this way, we have essentially converted the variable domain case to the simpler fixed domain case.

\section{Decoder Width Estimation for Several Parametric PDEs}\label{sec:decWidthEst}
In this section, we shall examine several common parametric PDEs, and estimate the decoder widths of the corresponding solution sets. 
The proofs of these assertions can be found in Section~\ref{sec:proofs}. 

Before diving into the rigorous statements and proofs, we shall first give a brief informal description of the intuition behind. 
For the elliptic and the parabolic equations on a fixed domain, the solution mappings are Lipschitz continuous, i.e. $\|u-\bar u\|\le C\|\eta-\bar\eta\|$ holds. 
By Proposition~\ref{prop:decWidthLip}, this would imply $d_{n,Cl}^\text{Deco}(\{u^\eta\})\le Cd_{n,l}^\text{rDeco}(\{\eta\})$. 
When the specific form of $\mathcal{A}=\{\eta\}$ is given (in the examples), the specific upper bounds for $d_{n,Cl}^\text{Deco}(\{u^\eta\})$ can be derived accordingly. 

In terms of the scenario involving a variable domain, we only present the results for the elliptic equations, and the case of the parabolic equations could be tackled in a similar manner. 
The overall strategy for deriving upper bounds remains consistent with the fixed-domain case, 
but additional efforts are required to establish the Lipschitz continuity of the solution mapping, since the variable PDE parameter $\eta$ now incorporates the shape of the computational domain $\Omega$.  
To address this, we first examine how the solution $u$ is affected by a small perturbation of $\Omega$ (Section~\ref{sec:ellip2domain}). 
A master domain $\Omega^\text{m}$ is predefined, and the perturbed domain $\Omega_2=\beta(\Omega)$ is given by an automorphism $\beta:\Omega^\text{m}\to\Omega^\text{m}$ that is close to the identity. 
Subsequently, we compare solutions on two arbitrary domains by introducing a sequence of intermediate domains, and conclude that the solution mapping is indeed Lipschitz continuous (Proposition~\ref{prop:eVDtransLip}). 

The case appears to be slightly different for the advection equations on a fixed domain, as the solution mapping is Lipschitz only when the function space $\mathcal{U}$ is endowed with the $L^1$-norm. 
If we switch to the $L^2$-norm that is more convenient for numerical computation, 
the solution mapping becomes only $\frac12$-H\"older continuous, i.e. $\|u-\bar u\|\le C\sqrt{\|\eta-\bar\eta\|}$. 
We would then recourse to the entropy number as the alternative tool for analysis, and show $\epsilon_n(\{u^\eta\})\le C\sqrt{\tilde\epsilon_n(\{\eta\})}$ (Lemma~\ref{lem:ent_Holder}). 
Since $d^\text{Deco}_{26n,4\delta}(\{u^\eta\})\le 3\epsilon_n(\{u^\eta\})$ holds according to Corollary~\ref{cor:decoVSentNum}, the specific form of the upper bounds can be established. 
\subsection{Elliptic Equations on a Fixed Domain}
Let $\Omega\subset\R^d$ be a bounded domain with Lipschitz continuous boundary, 
we consider the elliptic equation
\begin{equation}\label{eq:ellipFixDom}\begin{aligned}
	Lu&=f&\text{in }\Omega,
	\\u&=0&\text{on }\partial\Omega
,\end{aligned}\end{equation}
where the differential operator $L$ is defined as
\[
	Lu=-\nabla\cdot(\pmb{a}(\pmb{x})\nabla u(\pmb{x}))+c(\pmb{x})u(\pmb{x})
\]
with variable coefficients $\pmb{a}\in\Sigma^d_{r,R}(\Omega)$, $c\in \mathsf{D}^c$, $f\in \mathsf{D}_R^f$. 
Here
\begin{equation}\label{eq:defSigma_drR}\begin{split}
	\Sigma^d_{r,R}(\Omega)=\{\pmb{a}(\pmb{x})\mid&r\xi^\mathrm{T}\xi\le\xi^\mathrm{T}\pmb{a}(\pmb{x})\xi\le R\xi^\mathrm{T}\xi,
	\\&\pmb{a}(\pmb{x})^\mathrm{T}=\pmb{a}(\pmb{x})\in\R^{d\times d},\forall\xi\in\R^d, \pmb{x}\in\Omega\}
\end{split}\end{equation}
is a subset of $L^\infty(\Omega;\ell^2(\R^{d\times d}))$ with $0<r<R<\infty$, 
\[\mathsf{D}^c=\{c(\pmb{x})\in L^\infty(\Omega)\mid c(\pmb{x})\ge 0,\forall \pmb{x}\in\Omega\} ,\]
and $\mathsf{D}_R^f=\bar B_R(L^2(\Omega))$ using the notation~\eqref{eq:defBR}. 
The weak solution $u\in H^1_0(\Omega)$ to~\eqref{eq:ellipFixDom} should satisfy
\begin{equation}\label{eq:ellipFixDomWeak}
	\int_{\Omega}\left((\nabla u)^\mathrm{T}\pmb{a}\nabla v+cuv\right)\,\mathrm{d}\pmb{x}
	=\int_{\Omega}fv\,\mathrm{d}\pmb{x},\quad\forall v\in H^1_0(\Omega)
.\end{equation}
We have the following proposition:
\begin{prop}\label{prop:ellpLip}
	Let $\bar u\in H_0^1(\Omega)$ be the weak solution to~\eqref{eq:ellipFixDom}
	with alternative coefficients $\pmb{\bar a}\in\Sigma^d_{r,R}(\Omega)$, $\bar c\in \mathsf{D}^c$, $\bar f\in \mathsf{D}_R^f$. 
	Then we shall have
	\[\begin{split}
		\|u-\bar u\|_{H_0^1}
		\le \frac{C_\text{P}}{r^2}\Bigl(
			R\|\pmb{\bar a}-\pmb{a}\|_{L^\infty(\Omega;\ell^2(\R^{d\times d}))}
			+RC_\text{P}^2\|\bar c-c\|_{L^\infty}
		+r\|\bar f-f\|_{L^2} \Bigr)
	,\end{split}\]
	where $C_\text{P}$ is the Poincar\'e constant of $\Omega$. 
\end{prop}
\verB{
If we denote $\mathcal{A}_0=\Sigma^d_{r,R}(\Omega)\times \mathsf{D}^c\times \mathsf{D}_R^f$ with metric%
\footnote{This metric can be viewed as inherited from the norm of a Banach space. Therefore, Proposition~\ref{prop:decWidthLip} is applicable. }
given as
\[d_{\mathcal{A}_0}((\pmb{a},c,f),(\pmb{\bar a},\bar c,\bar f)) =
	R\|\pmb{\bar a}-\pmb{a}\|_{L^\infty(\Omega;\ell^2(\R^{d\times d}))}
	+RC_\text{P}^2\|\bar c-c\|_{L^\infty}
	+r\|\bar f-f\|_{L^2}
,\]
then Proposition~\ref{prop:ellpLip} indicates that
the (extended) solution mapping
$ S_0:\mathcal{A}_0 \to H_0^1(\Omega), \quad(\pmb{a},c,f)\mapsto u $
is $\frac{C_\text{P}}{r^2}$-Lipschitz continuous. 
The following theorem is then direct from Proposition~\ref{prop:decWidthLip}:
\begin{thm}\label{thm:ellipDecoKvsA}
	Let $\mathcal{A}\subset\mathcal{A}_0$ be the set of parameters taken into consideration. 
	Then we shall have
	\[d_{n,C_\text{P}l/r^2}^\text{Deco}(S_0(\mathcal{A}))\le\frac{C_\text{P}}{r^2}d_{n,l}^\text{rDeco}(\mathcal{A}) .\]
	Intuitively speaking, the decoder width of the solution set $S_0(\mathcal{A})$ is controlled by the internal complexity (measured by the restricted decoder width) of the parameter set $\mathcal{A}$. 
\end{thm}
}
One simple special case is given as follows: 
\begin{eg}\label{eg:ellipEg1}
	Let $\Omega=(0,1)^2$ be a square domain, $r=1,R=2$, and fix $c\in \mathsf{D}^c$, $f\in \mathsf{D}_2^f$. 
	Given a matrix $\lambda\in\Lambda=[1,2]^{K\times K}$, we take $\pmb{a}^\lambda=(\sum_{i,j=1}^K\lambda_{ij}\chi_{D_{ij}})I_2$ to be a piecewise constant field, 
	where $D_{ij}=[\frac{i-1}{K},\frac iK)\times[\frac{j-1}{K},\frac jK)$ 
	(see Figure~\ref{fig:grid_field}). 
	Let $u^\lambda\in \mathcal{U}=H_0^1(\Omega)$ be the weak solution to~\eqref{eq:ellipFixDom} with coefficients $(\pmb{a}^\lambda,c,f)$. 
	Then we shall have $d_{K^2,KC_\text{P}}^\text{Deco}(\{u^\lambda|\lambda\in\Lambda\})=0$, 
	which indicates that the solution set can be perfectly fitted using a $K^2$-dimensional latent space. 
\end{eg}
\begin{figure}[tb]
	\centering
	\includegraphics[width=0.2\linewidth]{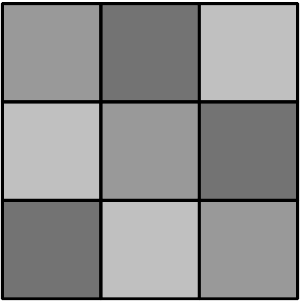}
	\caption{The piecewise constant field $\pmb{a}^\lambda$ used in Example~\ref{eg:ellipEg1}, showing the case $K=3$. Note that $\pmb{a}^\lambda$ is in fact a matrix-valued field. }%
	\label{fig:grid_field}
\end{figure}

\subsection{Parabolic Equations on a Fixed Domain}\label{sec:parabEqn}
Letting $\Omega\subset\R^d$ be a bounded domain with Lipschitz continuous boundary, $0<r<R<\infty$, $T>0$, we consider
the following parabolic equation on $\Omega_T=[0,T]\times\Omega$:
\begin{equation}\label{eq:parabEqOrig}\begin{aligned}
	\partial_tu+Lu &=f\quad&\text{in }\Omega_T
	\\u&=0&\text{on }[0,T]\times\partial\Omega
	\\u&=g&\text{on }\{t=0\}\times\Omega
.\end{aligned}\end{equation}
Here, we define the differential operator $L$ in the divergence form to be
\[
	Lu=-\nabla\cdot(\pmb{a}\nabla u)+\pmb{b}\cdot\nabla u+cu
\]
with variable coefficients
$\pmb{a}\in\Sigma^d_{r,R}(\Omega_T)$, $\pmb{b}\in\bar B_R^d(\Omega_T)$, $c\in\bar B_R^1(\Omega_T)$, where 
\[\begin{split}
	\bar B_R^d(\Omega_T)&=\bar B_R(L^\infty(\Omega_T;\ell^2(\R^d)))
	\\&=\left\{\pmb{b}\in L^\infty(\Omega_T;\R^d)\ \middle|\ \operatorname*{ess\,sup}_{t\in[0,T],\pmb{x}\in\Omega}\|\pmb{b}(t,\pmb{x})\|_2\le R\right\}
\end{split}\]
using the notation given in~\eqref{eq:defBR}, 
and similarly
\[\begin{split}
	\bar B_R^1(\Omega_T)&=\bar B_R(L^\infty(\Omega_T))
	\\&=\left\{c\in L^\infty(\Omega_T)\ \middle|\ \|c\|_{L^\infty}\le R\right\}
.\end{split}\]
Furthermore, we assume
$g\in L^2(\Omega)$, and $f(t,\pmb{x})=\mathbf{f}(t)(\pmb{x})$ for some $\mathbf{f}\in L^2([0,T];H^{-1}(\Omega))$. 
The solution $u$ can also be associated with a mapping $\mathbf{u}$ by $u(t,\pmb{x})=\mathbf{u}(t)(\pmb{x})$. 

We say a function $\mathbf{u}\in L^2([0,T];H_0^1(\Omega))$ with $\mathbf{u}'=\frac{\mathrm{d}}{\mathrm{d}t}\mathbf{u}\in L^2([0,T];H^{-1}(\Omega))$ is a weak solution to~\eqref{eq:parabEqOrig} if
\begin{equation}\label{eq:parabEqWeak}\begin{aligned}
	\langle \mathbf{u}'(t),v\rangle+A[\mathbf{u}(t),v;t]&=\langle \mathbf{f}(t),v\rangle,
	\\\mathbf{u}(0)&=g
\end{aligned}\end{equation}
holds\footnote{According to Theorem~3\ in Section~5.9.2 of~\cite{Evans2010PartialDE}, 
we shall have $\mathbf{u}\in C([0,T];L^2(\Omega))$, and therefore the equality $\mathbf{u}(0)=g$ makes sense. }
for all $v\in H_0^1(\Omega),t\in[0,T]$, 
where the time-dependent bilinear form $A[u,v;t]$ is given by
\[\begin{split}
	A[u,v;t]&=
	\int_\Omega(\nabla v(\pmb{x}))^\mathrm{T}\pmb{a}(t,\pmb{x})(\nabla u(\pmb{x}))+v(\pmb{x})\pmb{b}(t,\pmb{x})^\mathrm{T}\nabla u(\pmb{x})+c(t,\pmb{x})u(\pmb{x})v(\pmb{x})\,\mathrm{d}\pmb{x}  
\end{split}\]
for $u,v\in H_0^1(\Omega)$ and $t\in[0,T]$.
\begin{prop}\label{prop:parabSLip}
	Denote $\mathsf{D}_R^\mathbf{f}=\bar B_R(L^2([0,T];H^{-1}(\Omega)))$, $\mathsf{D}_R^g=\bar B_R(L^2(\Omega))$ using the notation~\eqref{eq:defBR}, and
	\[\mathcal{A}_0=
		\Sigma^d_{r,R}(\Omega_T)\times\bar B_{R}^{d}(\Omega_T)\times\bar B_{R}^1(\Omega_T)
		\times \mathsf{D}_R^\mathbf{f}\times \mathsf{D}_R^g
	.\]
	Then the (extended) solution mapping
	\[\begin{split}
		S_0:\mathcal{A}_0
		\to L^2([0,T];H_0^1(\Omega)),
		\quad(\pmb{a},\pmb{b},c,\mathbf{f},g)\mapsto\mathbf{u}
	\end{split}\]
	is Lipschitz continuous. 
	Specifically, 
	let $(\pmb{a},\pmb{b},c,\mathbf{f},g)$ and $(\pmb{\bar a},\pmb{\bar b},\bar c,\bar{\mathbf{f}},\bar g)$ be two sets of coefficient fields, with the corresponding weak solutions $\mathbf{u}$ and $\bar{\mathbf{u}}$, respectively. 
	Then we should have
	\begin{equation}\label{eq:parabLipThm}\begin{split}
		\|\mathbf{u}-\bar{\mathbf{u}}\|_{L^2([0,T];H_0^1(\Omega))}
		&\le C(\|\pmb{a}-\pmb{\bar a}\|_{L^\infty(\Omega_T,\ell^2(\R^{d\times d}))}
		+\|\pmb{b}-\pmb{\bar b}\|_{L^\infty(\Omega_T,\ell^2(\R^d))}
		\\&\quad+\|c-\bar c\|_{L^\infty(\Omega_T)}
		+\|\mathbf{f}-\bar{\mathbf{f}}\|_{L^2([0,T];H^{-1}(\Omega))}
		\\&\quad+\|g-\bar g\|_{L^2})
	,\end{split}\end{equation}
	where the constant $C>0$ depends only on $r,R,T$ and $\Omega$. 
\end{prop}
\verB{
The metric on $\mathcal{A}_0$ is chosen according to the right-hand-side of~\eqref{eq:parabLipThm} (inside the parenthesis). 
Similar to the elliptic equations, the decoder width of the solution set admits an upper bound given by the complexity of the parameter set $\mathcal{A}$, as the following theorem would imply:
\begin{thm}
	Let $\mathcal{A}\subset\mathcal{A}_0$ be the set of parameters taken into consideration. Then 
	$d_{n,Cl}^\text{Deco}(S_0(\mathcal{A}))\le Cd_{n,l}^\text{rDeco}(\mathcal{A}) ,$
	where $C$ is the same as in Proposition~\ref{prop:parabSLip}.
\end{thm}
This theorem is easily obtained by combining Proposition~\ref{prop:decWidthLip} with the $C$-Lipschitz continuity of $S_0$ (Proposition~\ref{prop:parabSLip}), and we shall omit the proof. 
}
\subsubsection{Special Cases}
Assume we are given a sequence of functions
$\{(\pmb{\phi}_k^a,\pmb{\phi}_k^b,\phi_k^c,\phi_k^\mathbf{f},\phi_k^g)\}_{k=1}^\infty,$ 
satisfying: 
\begin{enumerate}[(a)]
	\item (boundedness) $\pmb{\phi}_k^a\in\bar B_{(R-r)/2}^d(\Omega_T)$, $\pmb{\phi}_k^b\in\bar B_R^d(\Omega_T)$, $\phi_k^c\in\bar B_R^1(\Omega_T)$. 
	\item (disjoint support) For $k\ne l$, we have $\mathrm{supp}\pmb{\phi}_k^a\cap\mathrm{supp}\pmb{\phi}_l^a=\varnothing$, where $\mathrm{supp}\pmb{\phi}_k^a$ represents the 
		essential support of the function $\pmb{\phi}_k^a:\Omega_T\to\R^d$. 
		Similar assumptions are made for $\pmb{\phi}_k^b$ and $\phi_k^c$. 
	\item $\phi_k^\mathbf{f}\in \mathsf{D}_R^\mathbf{f}$, $\sum_{k=1}^{\infty}\|\phi_k^\mathbf{f}\|_{L^2([0,T];H^{-1}(\Omega))}^2\le R^2$. 
	\item $\phi_k^g\in \mathsf{D}_R^g$, $\sum_{k=1}^{\infty}\|\phi_k^g\|_{L^2}^2\le R^2$. 
\end{enumerate}
Let
$\lambda=\{(\lambda_k^a,\lambda_k^b,\lambda_k^c,\lambda_k^\mathbf{f},\lambda_k^g)\}_{k=1}^\infty$ 
be a sequence of scalars, 
and we denote
\[ \lambda_k^2=(\lambda_k^a)^2+(\lambda_k^b)^2+(\lambda_k^c)^2+(\lambda_k^\mathbf{f})^2+(\lambda_k^g)^2 \]
for simplicity. 
We further assume that the corresponding norm
\[ \|\lambda\|=\sqrt{\sum_{k=1}^{\infty}\lambda_k^2} \]
is finite. 
The mapping $S_\Lambda:\lambda\mapsto(\pmb{a},\pmb{b},c,\mathbf{f},g)$ is defined as
\[\begin{split}
	\pmb{a}&=\frac12(R+r)I_d+\mathrm{diag}\biggl(\sum_{k=1}^{\infty}\lambda_k^a\pmb{\phi}_k^a\biggr),
	\quad \pmb{b}=\sum_{k=1}^{\infty}\lambda_k^b\pmb{\phi}_k^b,
	\\c&=\sum_{k=1}^{\infty}\lambda_k^c\phi_k^c,
	\quad \mathbf{f}=\sum_{k=1}^{\infty}\lambda_k^\mathbf{f}\phi_k^\mathbf{f},
	\quad g=\sum_{k=1}^{\infty}\lambda_k^g\phi_k^g
,\end{split}\]
\verB{and we denote $u^\lambda=S_0(S_\Lambda(\lambda))$ to be the weak solution to~\eqref{eq:parabEqOrig} for specific $\lambda$. }
Two special cases are given as follows:
\begin{eg}\label{eg:parabEg1}
	Let
	$ \Lambda_K=\{\lambda\mid\lambda_k^2=0\text{ for }k>K,\quad \sum_{k=1}^{K}\lambda_k^2\le 1\} ,$
	which essentially consists of sequences of finite length. 
	Then (1) $S_\Lambda(\lambda)\in\mathcal{A}_0$ holds for any $\lambda\in\Lambda_K$,
	(2) the mapping $S_0\circ S_\Lambda:\Lambda_K\to \mathcal{U}=L^2([0,T];H_0^1(\Omega))$ is $5CR$-Lipschitz continuous, where the constant $C$ is the same as in Proposition~\ref{prop:parabSLip}, 
	and (3) $d_{5K,5CR}^\text{Deco}(\mathcal{K}_K)=0$, 
	where $\mathcal{K}_K=\{u^\lambda\mid\lambda\in\Lambda_K\}=S_0\circ S_\Lambda(\Lambda_K)$ is the solution set. 
\end{eg}
\begin{eg}\label{eg:parabEg2}
	Let $\{w_k\}$ be a decreasing sequence with $w_1=1$, $\lim_{k\to\infty}w_k=0$, and let
	$ \Lambda_w=\{\lambda\mid\sum_{k=1}^\infty\lambda_k^2/w_k\le 1\} .$
	Then the solution set $\mathcal{K}_w=\{u^\lambda\mid\lambda\in\Lambda_w\}=S_0\circ S_\Lambda(\Lambda_w)$ satisfies
	$d_{5n,5CR}^\text{Deco}(\mathcal{K}_w)\le 5CR\sqrt{w_n}$, 
	where the constant $C$ is the same as in Proposition~\ref{prop:parabSLip}. 
\end{eg}

\subsection{Elliptic Equations on a Variable Domain}\label{sec:ellpVD}
Assume we are given a fixed \emph{convex} master domain $\Omega^\text{m}$, and the subdomain of interest $\Omega\subset\Omega^\text{m}$ may vary. 
All possible choices of $\Omega$ as well as $\Omega^\text{m}$
have a Lipschitz continuous boundary. 
In this case, the PDE parameter $\eta=(\pmb{a},f,\Omega)$ would involve the shape of the computational domain. We
consider the elliptic equation
\begin{equation}\label{eq:ellipRawDom}\begin{aligned}
	Lu&=f&\text{in }\Omega,
	\\u&=0&\text{on }\partial\Omega
,\end{aligned}\end{equation}
where
\[
	Lu=-\nabla\cdot(\pmb{a}(\pmb{x})\nabla u(\pmb{x}))=-\sum_{i,j=1}^d\partial_j(\pmb{a}^{ij}(\pmb{x})\partial_iu(\pmb{x}))
,\]
and the following assumptions are satisfied: 
\begin{assump}\label{ass:vardomEllip}
	(1) $\pmb{a}:\Omega^\text{m}\to\ell^2(\R^{d\times d})$ is Lipschitz continuous with Lipschitz constant $R$, 
	(2) $\pmb{a}\in\Sigma_{r,R}^d(\Omega^\text{m})$, i.e. for all $\pmb{x}\in\Omega^\text{m}$, $\pmb{a}(\pmb{x})$ is a symmetric matrix satisfying
	\[
		r\xi^\mathrm{T}\xi\le\xi^\mathrm{T}\pmb{a}(\pmb{x})\xi\le R\xi^\mathrm{T}\xi,\quad\forall\xi\in\R^d
	,\]
	and (3) $f\in\bar B_R(L^2(\Omega))$ (using the notation~\eqref{eq:defBR}), i.e. $f\in L^2(\Omega^\text{m})$, $\|f\|_{L^2(\Omega^\text{m})}\le R$. 
	
\end{assump}
In order to estimate the decoder width of the set of weak solutions $\{u^\eta\}$ to~\ref{eq:ellipRawDom}, we consider the following PDEs on the extended domain
\begin{equation}\label{eq:ellipVarDom}\begin{aligned}
	Lu&=f&\text{in }\Omega^\text{m}\setminus\partial\Omega,
	\\u&=0&\text{on }\partial\Omega^\text{m}\cup\partial\Omega
.\end{aligned}\end{equation}
Then, if $\bar u\in H^1_0(\Omega^\text{m}\setminus\partial\Omega)$ is the weak solution to~\eqref{eq:ellipVarDom}, i.e.,
\[
	\int_{\Omega^\text{m}}(\nabla \bar u(\pmb{x}))^\mathrm{T}\pmb{a}(\pmb{x})(\nabla v(\pmb{x}))\,\mathrm{d}\pmb{x}
	=\int_{\Omega^\text{m}}f(\pmb{x})v(\pmb{x})\,\mathrm{d}\pmb{x},\quad\forall v\in H^1_0(\Omega^\text{m}\setminus\partial\Omega)
,\]
it is easy to see that $u=\bar u|_\Omega$ is the weak solution to~\eqref{eq:ellipRawDom}, 
and we only need to examine the set of these extended solutions $\{\bar u^\eta\}$. 
For notational convenience, we do not explicitly distinguish the solutions to~\eqref{eq:ellipRawDom} and~\eqref{eq:ellipVarDom} unless specified otherwise, and denote $u$ to be the weak solution to~\eqref{eq:ellipVarDom}. 

Now we consider a class of domains $\Omega_\lambda\subset\Omega^\text{m}$ parametrized by $\lambda\in\Lambda$, where $\Lambda\subset\R^{d_\lambda}$ is a bounded convex set that contains the origin. 
The following assumptions are made on the parametrized domains:
\begin{assump}\label{ass:vardomTransforms}
	There exist a parametrized class of $C^1$-diffeomorphisms%
	\footnote{Explicit domain transformations are required in the theoretical estimation of the decoder width, but not in the numerical computation to solve the parametric PDEs.
	MAD is still a flexible algorithm in this sense. }
	$T_\lambda:\bar\Omega^\text{m}\to\bar\Omega^\text{m}$ satisfying: 
	\begin{enumerate}[(1)]
		\item $T_0=\mathrm{id}$, $T_\lambda(\Omega_0)=\Omega_\lambda$. 
		\item  $\sigma_{\min}(\frac{\partial T_\lambda(x)}{\partial x})\ge 1/2$ holds for any $x\in\bar\Omega^\text{m},\lambda\in\Lambda$, where $\sigma_{\min}$ denotes the smallest singular value of a matrix. 
		\item Given $x\in\bar\Omega^\text{m}$, both $T_\lambda(x)$ and $\frac{\partial T_\lambda(x)}{\partial x}$ are $1$-Lipschitz continuous with respect to $\lambda$, i.e.,
			\[\|T_{\lambda}(x)-T_{\lambda'}(x)\|_2\le\|\lambda-\lambda'\|_2,\quad\left\lVert \frac{\partial T_{\lambda}(x)}{\partial x}-\frac{\partial T_{\lambda'}(x)}{\partial x}\right\rVert_2\le\|\lambda-\lambda'\|_2 \]
			for any $\lambda,\lambda'\in\Lambda$. 
	\end{enumerate}
\end{assump}

\verB{
\subsubsection{Fixed Parameters}
In the first case, the parameters $\pmb{a},f$ are fixed for simplicity, and
}
the weak solution to~\eqref{eq:ellipVarDom} with domain $\Omega=\Omega_\lambda$ is denoted as $u^\lambda\in{H_0^1(\Omega^\text{m}\setminus\partial\Omega_\lambda)}$. 
We have the following proposition on the corresponding solution mapping:
\begin{prop}\label{prop:eVDtransLip}
	Under Assumption~\ref{ass:vardomEllip} and~\ref{ass:vardomTransforms}, the solution mapping $S:\Lambda\to \mathcal{U}=L^2(\Omega^\text{m})$, $\lambda\mapsto u^\lambda$ is $C$-Lipschitz continuous, 
	where the constant $C$ depends only on $r,R$ and $\Omega^\text{m}$. 
	Specifically, we can take
	\[C=\frac{4RC_\text{P}}{r^2}(3r+(\tfrac38(9+\delta_0^{-1})+1)RC_\text{P}),\]
	where $C_\text{P}$ is the Poincar\'e constant of $\Omega^\text{m}$, 
	and $\delta_0=\min(1-(8/9)^{2/d},1/9)$. 
\end{prop}
The following theorem then follows from Proposition~\ref{prop:decWidthLip}: 
\begin{thm}\label{thm:ellipVarDomDecWidthLip}
	Under Assumption~\ref{ass:vardomEllip} and~\ref{ass:vardomTransforms}, 
	we have
	$d_{n,Cl}^\text{Deco}(S(\Lambda))\le Cd_{n,l}^\text{rDeco}(\Lambda) ,$
	where the constant $C$ is the same as in Proposition~\ref{prop:eVDtransLip}. 
\end{thm}

\begin{eg}[moving disk]\label{eg:eVDmovingDisk}
	We take the master domain to be the two dimensional unit disk $\Omega^\text{m}=\{\pmb{x}\in\R^2\mid\|\pmb{x}\|_2<1\}$. 
	Given $\lambda\in\Lambda=[-\pi,\pi]$, the corresponding domain is
	\[\Omega_\lambda=\{\pmb{x}\in\R^2\mid(x_1-\tfrac13\cos\lambda)^2+(x_2-\tfrac13\sin\lambda)^2<\tfrac19\},\]
	which rotates around the origin as $\lambda$ varies. 
	Then $d_{1,C\pi}^\text{Deco}(\{u^\lambda|\lambda\in\Lambda\})=0$ ($C$ being the same as in Proposition~\ref{prop:eVDtransLip}), indicating that the solution set can be perfectly fitted using a one-dimensional latent space. 
\end{eg}
\begin{eg}[unit disk with a moving hole]\label{eg:eVDmovingHole}
	We take $\Omega^\text{m}=\{\pmb{x}\in\R^2\mid\|\pmb{x}\|_2<1\}$, $\lambda\in\Lambda=[-\pi,\pi]$ as in the previous example, but replace $\Omega_\lambda$ by the interior of its complement, namely $\Omega_\lambda=\{\pmb{x}\in\Omega^\text{m}\mid(x_1-\frac13\cos\lambda)^2+(x_2-\frac13\sin\lambda)^2>\frac19\}$. 
	The resulting domain is the unit disk with a hole inside, and the hole changes its position as $\lambda$ varies. 
	Then we shall have $d_{1,C\pi}^\text{Deco}(\{u^\lambda|\lambda\in\Lambda\})=0$ as in the previous example. 
\end{eg}
\begin{eg}[square with a deformable rectangular hole]\label{eg:eVDdefmHole}
	The master domain is taken to be $\Omega^\text{m}=(-\pi,\pi)^2$, and we obtain $\Omega_\lambda$ by hollowing out a rectangular hole in the form
	\[ \Omega^\text{m}\setminus\Omega_\lambda
		=\left[-\frac\pi2-\lambda_1,\frac\pi2+\lambda_1\right]\times\left[-\frac\pi2-\lambda_2,\frac\pi2+\lambda_2\right]
	,\]
	where $\lambda=(\lambda_1,\lambda_2)\in\Lambda=[-\frac12,\frac12]^2$ is the parameter. 
	Then we shall have $d_{2,C/\sqrt 2}^\text{Deco}(\{u^\lambda|\lambda\in\Lambda\})=0$. 
\end{eg}
\begin{eg}[area under a variable curve]\label{eg:eVDctvc}
	Let $\lambda=(a_1,b_1,a_2,\dots)$ be an infinite sequence of real numbers
	with norm defined as $\|\lambda\|^2=\sum_{k=1}^{\infty}(a_k^2+b_k^2)$, and
	\[\Lambda=\left\{\lambda\ \middle|\ \sum_{k=1}^{\infty}\frac{a_k^2+b_k^2}{w_k}\le\frac9{64}\right\} ,\]
	where $w_k$ is a decreasing sequence with $w_1=1$, $\lim_{k\to\infty}w_k=0$. 
	Letting
	\[\rho_\lambda(x_1)=\sum_{k=1}^{\infty}\frac{\sqrt 6}{k^2\pi}(a_k\cos kx_1+b_k\sin kx_1) ,\]
	we take $\Omega^\text{m}=(-\pi,\pi)\times(0,\pi)$, and
	\[\Omega_\lambda=\left\{\pmb{x}\in\R^2\middle|-\pi<x_1<\pi,0<x_2<\frac\pi 2+\rho_\lambda(x_1)\right\} .\]
	Then $d_{2n,3C/8}^\text{Deco}(\{u^\lambda|\lambda\in\Lambda\})\le 3C\sqrt{w_n}/8$. 
\end{eg}
\begin{figure}[tb!]
	\centering
	\begin{subfigure}[p]{0.3\textwidth}
		\centering
		\includegraphics[width=\textwidth]{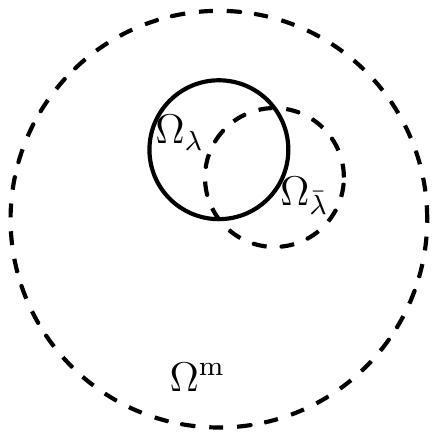}
	\end{subfigure}%
	~ 
	\begin{subfigure}[p]{0.25\textwidth}
		\centering
		\includegraphics[width=\textwidth]{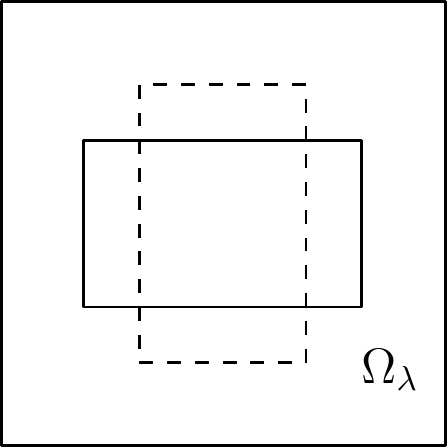}
	\end{subfigure}%
	~ 
	\begin{subfigure}[p]{0.35\textwidth}
		\centering
		\includegraphics[width=\textwidth]{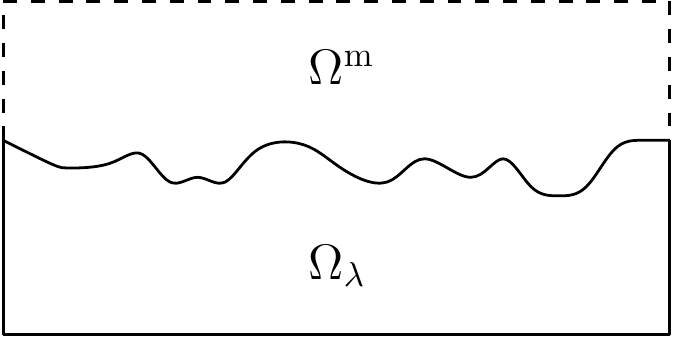}
	\end{subfigure}%
	~ 
	\caption{Illustrative plot of the $\Omega_\lambda$'s in Example~\ref{eg:eVDmovingDisk} (left), \ref{eg:eVDdefmHole} (middle), and \ref{eg:eVDctvc} (right). }%
	\label{figs:ellpVD_eg}
\end{figure}

\subsubsection{Variable Parameters}
Now both $\pmb{a}$ and $f$ in the PDEs~\eqref{eq:ellipVarDom} can vary along with the shape of the solution domain $\Omega$ (or its parameter $\lambda$). 
According to Assumption~\ref{ass:vardomEllip}, we denote $\mathsf{D}_{r,R}^a$ as the set of all possible $\pmb{a}$'s, and $\mathsf{D}_R^f=\bar B_R(L^2(\Omega))$ the set of all possible $f$'s. 
Then we shall have the following proposition: 
\begin{prop}\label{prop:eVDPLip}
	Let $\mathcal{A}_0=\mathsf{D}_{r,R}^a\times \mathsf{D}_R^f\times\Lambda$ with the metric given as
	\begin{equation}\label{eq:eVDPdefA0metric}
		d_{\mathcal{A}_0}((\pmb{a},f,\lambda),(\pmb{\bar a},\bar f,\bar\lambda)) =
		\|\pmb{\bar a}-\pmb{a}\|_{L^\infty(\Omega^\text{m};\ell^2(\R^{d\times d}))}
		+\|\bar f-f\|_{L^2}
		+\|\bar\lambda-\lambda\|_2
	.\end{equation}
	Then, under Assumption~\ref{ass:vardomTransforms}, the (extended) solution mapping $S_0:\mathcal{A}_0\to\mathcal{U}=L^2(\Omega^\text{m})$, $(\pmb{a},f,\lambda)\mapsto u^{(\pmb{a},f,\lambda)}$ is $C_1$-Lipschitz continuous, where
	$C_1=\max\left(\frac{C_\text{P}^2R}{r^2},C\right)$, and $C$ is the same as in Proposition~\ref{prop:eVDtransLip}
\end{prop}
Again, the following theorem is direct according to Proposition~\ref{prop:decWidthLip}: 
\begin{thm}\label{thm:eVDPwidthFromA}
	Let $\mathcal{A}\subset\mathcal{A}_0$ be the set of parameters taken into consideration. 
	Then, under Assumption~\ref{ass:vardomTransforms}, we have
	$d_{n,C_1l}^\text{Deco}(S_0(\mathcal{A}))\le C_1d_{n,l}^\text{rDeco}(\mathcal{A}) ,$
	where the constant $C_1$ is the same as in Proposition~\ref{prop:eVDPLip}. 
\end{thm}
\begin{eg}\label{eg:eVPD}
	We take the domain to be a square with a deformable rectangular hole, and $\Omega^\text{m},\Omega_\lambda$, $\Lambda$ are the same as in Example~\ref{eg:eVDdefmHole}. 
	Taking $r=1,R=2$, we parametrize the PDE coefficients as $\pmb{a}=(\frac32+\sum_{k=1}^{K}\mu_k^a\phi_k^a)I_2$, $f=\sum_{k=1}^{K}\mu_k^f\phi_k^f$, where
	\begin{enumerate}[(a)]
		\item $\phi_k^a:\Omega^\text{m}\to\R$ are 4-Lipschitz continuous, $\|\phi_k^a\|_{L^\infty}\le 1$, $\phi_k^a$ and $\phi_l^a$ has disjoint supports for $1\le k<l\le K$; 
		\item $\phi_k^f\in L^2(\Omega^\text{m})$, $\sum_{k=1}^{K}\|\phi_k^f\|_{L^2}^2\le 1$; 
		\item $\mu=(\mu_k^a,\mu_k^f)_{k=1}^K\in\Lambda_\mu$, $\Lambda_\mu=\{\mu\mid\sum_{k=1}^{K}(\mu_k^a)^2\le\frac14,\sum_{k=1}^{K}(\mu_k^a)^2\le\frac14\}\subset\R^{2K}$. 
	\end{enumerate}
	Denote the corresponding solution as $u^{\mu,\lambda}$, we shall have
	$d_{2K+2,3C_1}^\text{Deco}(\{u^{\mu,\lambda}|\mu\in\Lambda_\mu,\lambda\in\Lambda\})=0,$ 
	with $C_1$ given as in Proposition~\ref{prop:eVDPLip}. 
\end{eg}

\subsection{Advection Equations on a Fixed Domain}\label{sec:advVC}
Consider the equation
\begin{equation}\label{eq:advEq}
	u_t+c(x)u_x=0
\end{equation}
on the spatial-temporal domain $(t,x)\in[0,1]\times[-1,1]$, with initial condition $u(0,x)=\chi_{[-1,0]}(x)$ and boundary condition $u(t,-1)=1$. 
Here, the variable coefficient is $c\in \mathcal{A}_0$, where 
\[\mathcal{A}_0=\bigl\{c\in C^1([-1,1])\mid r\le c(x)\le R,\forall x\in[-1,1]\bigr\} ,\]
and $0<r<R<\infty$ holds. 
We shall view $\mathcal{A}_0$ as a subset of the Banach space $L^1([-1,1])$ (inheriting its metric), and denote the (extended) solution mapping to be $S_0:c\mapsto u$. 
By the method of characteristics, it is easy to show that the weak solution to the PDEs is given by
\[u(t,x)=\chi_{\{t\le\int_0^x\frac{\mathrm{d}s}{c(s)}\}} =
	\begin{cases}
		1,&t\le\int_0^x\frac{\mathrm{d}s}{c(s)}
		\\0,&t>\int_0^x\frac{\mathrm{d}s}{c(s)}
	\end{cases} .\]
\begin{prop}\label{prop:advEg1Holder}
	Let $\bar u$ be the weak solution to~\eqref{eq:advEq} with alternative coefficient $\bar c\in\mathcal{A}_0$, i.e. $\bar u=S_0(\bar c)$. 
	Viewing $u,\bar u\in L^p([0,1]\times[-1,1])$ for $p\in[1,+\infty)$, we have
	\[\|u-\bar u\|_{L^p}\le
	r^{-2/p}\|c-\bar c\|_{L^1([-1,1])}^{1/p} .\]
\end{prop}
Due to the particularity of the advection equations, the decoder widths using $L^1$-norm ($p=1$) and $L^2$-norm ($p=2$) will be estimated separately. 
The decoder widths can achieve zero for finite $n$ in the former case, while in the latter case, only an exponential decay is guaranteed, as the Lipschitz continuity of the solution mapping is no longer valid. 
\subsubsection{The Case of Using $L^1$-norm}
Taking $p=1$ in Proposition~\ref{prop:advEg1Holder}, we apply Proposition~\ref{prop:decWidthLip} to get the following theorem:
\begin{thm}\label{thm:advL1DecKvsA}
	Let $\mathcal{A}\subset\mathcal{A}_0$ be the set of parameters taken into consideration, and $\mathcal{U}=L^1([0,1]\times[-1,1])$. 
	Then we shall have
	$d_{n,l/r^2}^\text{Deco}(S_0(\mathcal{A}))_{L^1}\le r^{-2}d_{n,l}^\text{rDeco}(\mathcal{A}) .$
\end{thm}
\begin{eg}\label{eg:advEg1_L1}
	Assume the velocity field $c(x)=c^\lambda(x)$ is parametrized as
	\[c^\lambda(x)=\phi_0(x)+\sum_{k=1}^{K}\lambda_k\phi_k(x),\quad\lambda\in\Lambda=[-1,1]^K ,\]
	where $\phi_0,\dots,\phi_K$ are all continuous functions, and are chosen appropriately such that $c^\lambda(x)\in C([-1,1];[r,R])$ holds for any $\lambda\in\Lambda$. 
	Let $u^\lambda=S_0(c^\lambda)\in\mathcal{U}$ be the weak solution corresponding to $\lambda$. 
	Then, the set of weak solutions $\mathcal{K}=\{u^\lambda\mid\lambda\in\Lambda\}$ satisfies
	$ d_{K,l/r^2}^\mathrm{Deco}(\mathcal{K})_{L^1}=0 ,$
	where $l=\sqrt{K\sum_{k=1}^{K}\|\phi_k\|_{L^1([-1,1])}^2}$ is a constant. 
\end{eg}
\subsubsection{The Case of Using $L^2$-norm}
When $p=2$, Proposition~\ref{prop:advEg1Holder} only indicates the $\frac12$-H\"older continuity of the (extended) solution mapping $S_0:c\mapsto u$, 
and an analogy of Theorem~\ref{thm:advL1DecKvsA} is not available. 
In this case, 
the decoder width can be estimated using the following theorem:
\begin{thm}\label{thm:advL2EntKvsA}
	Let $\mathcal{A}\subset\mathcal{A}_0$ be the set of parameters taken into consideration, and take $\mathcal{U}=L^2([0,1]\times[-1,1])$. 
	Then we shall have
	$d_{26n,8}^\text{Deco}(S_0(\mathcal{A}))_{L^2}\le 3r^{-1}\sqrt{\tilde\epsilon_n(\mathcal{A})} .$
\end{thm}
\begin{eg}\label{eg:advEg1}
	We consider the same solution set $\mathcal{K}$ as in Proposition~\ref{eg:advEg1_L1}, but within a different function space $\mathcal{U}=L^2([0,1]\times[-1,1])$. 
	Then
	$d_{26n,8}^\mathrm{Deco}(\mathcal{K})_{L^2}\le C\cdot 2^{-n/2K}$
	for sufficiently large $n$, where
	$C=6r^{-1}\sqrt K\left(\prod_{k=1}^{K}\|\phi_k\|_{L^1([-1,1])}\right)^{1/2K}$ 
	is a constant.
\end{eg}

\section{Proofs}\label{sec:proofs}
\subsection{Basic Properties of the Decoder Width}
\begin{proof}[Proof of Proposition~\ref{prop:decWidthLip}]
	Given an $l$-Lipschitz mapping $D:Z\to\mathcal{K_V}$, the mapping $\mathcal{F}\circ D:Z\to\mathcal{K}$ is $L_\mathcal{F}l$-Lipschitz. 
	The definition of the restricted decoder width gives
	\[\begin{split}
		d_{n,L_\mathcal{F}l}^\text{rDeco}(\mathcal{F(K_V)})
		&\le\sup_{v\in\mathcal{K_V}}\inf_{\pmb{z}\in Z_B}\|\mathcal{F}(v)-\mathcal{F}\circ D(\pmb{z})\|_\mathcal{U}
		\\&\le L_\mathcal{F}\sup_{v\in\mathcal{K_V}}\inf_{\pmb{z}\in Z_B}\|v-D(\pmb{z})\|_\mathcal{V}
	.\end{split}\]
	We conclude the proof by taking infimum over all such mappings $D$. 
\end{proof}
\begin{proof}[Proof of Proposition~\ref{prop:decoVSsmani}]
	Let $E:\mathcal{K}\to Z$ and $D:Z\to\mathcal{U}$ be two $l$-Lipschitz continuous mappings, 
	and fix an arbitrary element $v\in\mathcal{K}$. 
	We define a new mapping $\bar D:Z\to\mathcal{U}$ by $\bar D(\pmb{z})=D(l\delta_{\mathcal{K}}\pmb{z}+E(v))$. 
	Then $\bar D$ is $l^2\delta_{\mathcal{K}}$-Lipschitz continuous. 
	For given $u\in\mathcal{K}$, $\bar {\pmb{z}}=\frac1{l\delta_{\mathcal{K}}}(E(u)-E(v))$ satisfies
	\begin{equation*}
		\|\bar {\pmb{z}}\|_2=\frac1{l\delta_{\mathcal{K}}}\|E(u)-E(v)\|_2
		\le\frac1{\delta_{\mathcal{K}}}\|u-v\|_{\mathcal{U}}
		\le 1
	,\end{equation*}
	and
	\begin{equation*}
		\bar D(\bar {\pmb{z}})=D(l\delta_{\mathcal{K}}\bar {\pmb{z}}+E(v))
		=D(E(u)-E(v)+E(v))=D(E(u))
		.
	\end{equation*}
	This gives
	\begin{equation*}
		\inf_{\pmb{z}\in Z_B}\|u-\bar D(\pmb{z})\|_{\mathcal{U}}
		\le\|u-\bar D(\bar {\pmb{z}})\|_{\mathcal{U}}
		=\|u-D(E(u))\|_{\mathcal{U}}
		.
	\end{equation*}
	Taking supremum over all $u\in\mathcal{K}$, and then infimum over all possible mappings $E,D,\bar D$ would then give
	\begin{equation*}
		\begin{split}
			d_{n,l^2\delta_{\mathcal{K}}}^\mathrm{Deco}(\mathcal{K})&=\inf_{\bar D\ l^2\delta_{\mathcal{K}}\text{-Lip}}\sup_{u\in\mathcal{K}}\inf_{\pmb{z}\in Z_B}\|u-\bar D(\pmb{z})\|_\mathcal{U}
			\\&\le\inf_{E,D\ l\text{-Lip}}\sup_{u\in\mathcal{K}}\|u-D(E(u))\|_{\mathcal{U}}
			\\&=d_{n,l}^{\mathrm{SMani}}(\mathcal{K})
		.\qedhere\end{split}
	\end{equation*}
\end{proof}
\begin{proof}[Proof of the assertion in Example~\ref{eg:decoVSsmani1}]
	For any pair of $l$-Lipschitz continuous mappings $E:\mathcal{U}\to Z=\R$, $D:\R\to\mathcal{U}$, 
	the Borsuk-Ulam theorem asserts that there exists a point $v\in\mathcal{K}=S^1$ such that $E(v)=E(-v)$, and
	\[\begin{split}
		\sup_{u\in\mathcal{K}}\|u-D(E(u))\|_2&\ge\max(\|v-D(E(v))\|_2,\|-v-D(E(-v))\|_2)
		\\&\ge\frac12(\|v-D(E(v))\|_2+\|-v-D(E(-v))\|_2)
		\\&=\frac12(\|v-D(E(v))\|_2+\|v+D(E(v))\|_2)
		\\&\ge\frac12\|v-D(E(v))+v+D(E(v))\|_2
		\\&=\|v\|_2=1
	.\end{split}\]
	We conclude that $d_{n,l}^{\mathrm{SMani}}(\mathcal{K})\ge 1$ by taking infimum over all possible pairs $(E,D)$. 
	As for the decoder width, take $\bar D:\R\to\mathcal{U}$ as $z\mapsto(\cos\pi z,\sin\pi z)$. 
	Then $\bar D$ is $\pi$-Lipschitz continuous, and $\bar D(Z_B)=\bar D([-1,1])=\mathcal{K}$ holds, which implies $d_{1,\pi}^{\mathrm{Deco}}(\mathcal{K})=0$, and the first infimum in~\eqref{eq:decWidth} is achieved at $D=\bar D$. 
\end{proof}

\subsection{Elliptic Equations on a Fixed Domain}
\begin{lem}\label{lem:ellpH01u}
	The weak solution to~\eqref{eq:ellipFixDom} satisfies
	$\|u\|_{H_0^1}\le C_\text{P}R/r$, where $C_\text{P}$ is the Poincar\'e constant of $\Omega$. 
\end{lem}
\begin{proof}
	Take the test function to be $v=u\in H^1_0(\Omega)$, then
	\[\begin{split}
		\int_{\Omega}(\nabla u)^\mathrm{T}\pmb{a}\nabla u+cu^2\,\mathrm{d}\pmb{x}
		=\int_{\Omega}fu\,\mathrm{d}\pmb{x}
		\le\|f\|_{L^2}\|u\|_{L^2}
		\le RC_\text{P}\|u\|_{H_0^1}
	,\end{split}\]
	where we used $\|f\|_{L^2}\le R$ by the assumption $f\in\mathsf{D}_R^f$. 
	In the mean time, 
	\[\begin{split}
		\int_{\Omega}(\nabla u)^\mathrm{T}\pmb{a}\nabla u+cu^2\,\mathrm{d}\pmb{x}
		\ge\int_{\Omega}(\nabla u)^\mathrm{T}\pmb{a}\nabla u\,\mathrm{d}\pmb{x}
		\ge \int_{\Omega}r(\nabla u)^\mathrm{T}\nabla u\,\mathrm{d}\pmb{x}
		=r\|u\|_{H_0^1}^2
	,\end{split}\]
	since $\pmb{a}\in\Sigma^d_{r,R}(\Omega)$ and $c(\pmb{x})\ge 0$. 
	Combining these two inequalities would give $\|u\|_{H_0^1}\le C_\text{P}R/r$. 
\end{proof}
\begin{proof}[Proof of Proposition~\ref{prop:ellpLip}]
	Similar to~\eqref{eq:ellipFixDomWeak}, the alternative weak solution $\bar u\in H_0^1(\Omega)$ should satisfy
	\begin{equation}\label{eq:ellipFixDomWeakBarU}
		\int_{\Omega}\left((\nabla \bar u)^\mathrm{T}\pmb{\bar a}\nabla v+\bar c\bar uv\right)\,\mathrm{d}\pmb{x}
		=\int_{\Omega}\bar fv\,\mathrm{d}\pmb{x}
	.\end{equation}
	for any $v\in H_0^1(\Omega)$. Subtracting~\eqref{eq:ellipFixDomWeakBarU} from~\eqref{eq:ellipFixDomWeak} would give
	\[\begin{split}
		&\quad\int_{\Omega}(\nabla u-\nabla\bar u)^\mathrm{T}\pmb{a}\nabla v+c(u-\bar u)v\,\mathrm{d}\pmb{x}
		\\&=\int_{\Omega}(\nabla\bar u)^\mathrm{T}(\pmb{\bar a}-\pmb{a})\nabla v+(\bar c-c)\bar uv+(f-\bar f)v\,\mathrm{d}\pmb{x}
	.\end{split}\]
	Taking $v=w=u-\bar u$, we have
	\begin{equation}\label{eq:n6kf0t}\begin{split}
		&\quad\int_{\Omega}(\nabla\bar u)^\mathrm{T}(\pmb{\bar a}-\pmb{a})\nabla w\,\mathrm{d}\pmb{x}
		\\&\le\int_{\Omega}\|\nabla\bar u\|_2\cdot\|\pmb{\bar a}(x)-\pmb{a}(x)\|_2\cdot\|\nabla w\|_2\,\mathrm{d}\pmb{x}
		\\&\le\|\pmb{\bar a}-\pmb{a}\|_{L^\infty(\Omega;\ell^2(\R^{d\times d}))}\int_{\Omega}\|\nabla\bar u\|_2\cdot\|\nabla w\|_2\,\mathrm{d}\pmb{x}
		\\&\le\|\pmb{\bar a}-\pmb{a}\|_{L^\infty(\Omega;\ell^2(\R^{d\times d}))}\|\bar u\|_{H_0^1}\|w\|_{H_0^1}
		\\&\le\frac{C_\text{P}R}{r}\|\pmb{\bar a}-\pmb{a}\|_{L^\infty(\Omega;\ell^2(\R^{d\times d}))}\|w\|_{H_0^1}
		\quad\text{by Prop.\ref{lem:ellpH01u}}
	,\end{split}\end{equation}
	\[\begin{split}
		&\quad\int_{\Omega}(\bar c-c)\bar uw+(f-\bar f)w\,\mathrm{d}\pmb{x}
		\\&\le\|\bar c-c\|_{L^\infty}\|\bar u\|_{L^2}\|w\|_{L^2}
			+\|\bar f-f\|_{L^2}\|w\|_{L^2}
		\\&\le C_\text{P}^2\|\bar c-c\|_{L^\infty}\|\bar u\|_{H_0^1}\|w\|_{H_0^1}
			+C_\text{P}\|\bar f-f\|_{L^2}\|w\|_{H_0^1}
		\\&\le \frac{C_\text{P}}r\Bigl(
			RC_\text{P}^2\|\bar c-c\|_{L^\infty}
			+r\|\bar f-f\|_{L^2}
			\Bigr)\|w\|_{H_0^1}
	,\end{split}\]
	and thus
	\[\begin{split}
		&\quad\int_{\Omega}(\nabla w)^\mathrm{T}\pmb{a}\nabla w+cw^2\,\mathrm{d}\pmb{x}
		\\&\le \frac{C_\text{P}}r\Bigl(
			R\|\pmb{\bar a}-\pmb{a}\|_{L^\infty(\Omega;\ell^2(\R^{d\times d}))}
			+RC_\text{P}^2\|\bar c-c\|_{L^\infty}
		\\&\hspace{4em}+r\|\bar f-f\|_{L^2}
			\Bigr)\|w\|_{H_0^1}
	.\end{split}\]
	The condition $\pmb{a}\in\Sigma^d_{r,R}(\Omega)$ and $c(\pmb{x})\ge 0$ further guarantees
	\[
		\int_{\Omega}(\nabla w)^\mathrm{T}\pmb{a}\nabla w+cw^2\,\mathrm{d}\pmb{x}
		\ge r\int_{\Omega}(\nabla w)^\mathrm{T}\nabla w\,\mathrm{d}\pmb{x}
		=r\|w\|_{H_0^1}^2
	,\]
	and we conclude the proof by combining the two inequalities. 
\end{proof}
If $Z_\Lambda$ is a convex compact subset of $Z=\R^n$ (with $\ell^2$-norm by default), the metric projection $P_{Z_\Lambda}:Z\to Z_\Lambda$ is given by 
\begin{equation}\label{eq:metricProjPK}
	P_{Z_\Lambda}(z)=\operatorname*{arg\,min}_{\bar z\in Z_\Lambda}\|\bar z-z\|_2
,\end{equation}
which is 1-Lipschitz continuous. 
\begin{proof}[Proof of the assertion in Example~\ref{eg:ellipEg1}]
	We first show that $d_{K^2,K/2}^\text{rDeco}(\Lambda)=0$ holds. 
	Let $Z=\R^{K^2}$, $Z_\Lambda=[-1/K,1/K]^{K^2}$. 
	The mapping
	\[D_1:Z_\Lambda\to(\R^{K\times K},\|\cdot\|_F),\quad z\mapsto\lambda\]
	given as
	$\lambda_{ij}=\frac32+\frac K2z_{Ki+j-K}$
	is $\frac K2$-Lipschitz continuous, and 
	$D_1(Z_\Lambda)=\Lambda$ holds. 
	Take $P_{Z_\Lambda}:Z\to Z_\Lambda$ to be the metric projection given by~\ref{eq:metricProjPK}. 
	Then $D=D_1\circ P_{Z_\Lambda}$ is also $\frac K2$-Lipschitz continuous, and $D(Z_B)=\Lambda$ since $Z_\Lambda\subset Z_B$. 
	The assertion $d_{K^2,K/2}^\text{rDeco}(\Lambda)=0$ then follows from the definition of the restricted decoder width. 

	Now we consider the variable coefficients of the PDEs. 
	It is direct to check that $\pmb{a}^\lambda\in\Sigma_{1,2}^2(\Omega)$ holds for any $\lambda\in\Lambda$. 
	Given $\lambda,\bar\lambda\in\Lambda$, we have
	\[\begin{split}
		\|\pmb{a}^\lambda-\pmb{a}^{\bar\lambda}\|_{L^\infty(\Omega;\ell^2(\R^{2\times 2}))}
		&=\operatorname*{ess\,sup}_{\pmb{x}\in\Omega}\left\lVert\biggl(\sum_{i,j=1}^{K}(\lambda_{ij}-\bar\lambda_{ij})\chi_{D_{ij}}(\pmb{x})\biggr)I_2\right\rVert_2
		\\&=\operatorname*{ess\,sup}_{\pmb{x}\in\Omega}\left\lvert\sum_{i,j=1}^{K}(\lambda_{ij}-\bar\lambda_{ij})\chi_{D_{ij}}(\pmb{x})\right\rvert
		\\&=\max_{i,j}|\lambda_{ij}-\bar\lambda_{ij}|\quad\text{(as }D_{ij}\text{'s are disjoint)}
		\\&\le\|\lambda-\bar\lambda\|_F
	,\end{split}\]
	and the mapping $S_\Lambda:\Lambda\to\mathcal{A}_0, \lambda\mapsto(\pmb{a}^\lambda,c,f)$ is therefore $R$-Lipschitz if we endow $\Lambda$ with the matrix Frobenius norm $\|\cdot\|_F$. 
	By Proposition~\ref{prop:decWidthLip}, the set of possible parameters $\mathcal{A}=S_\Lambda(\Lambda)$ satisfies $d_{K^2,KR/2}^\text{rDeco}(\mathcal{A})=0$. 
	We then apply Theorem~\ref{thm:ellipDecoKvsA}, and insert $r=1,R=2$ to conclude the proof. 
\end{proof}

\subsection{Parabolic Equations on a Fixed Domain}
\begin{lem}\label{lem:weakBilForm}
	For all $u\in H_0^1(\Omega)$, $t\in[0,T]$,
	\[
		\frac r2\|u\|_{H_0^1(\Omega)}^2\le A[u,u;t]+C_B\|u\|_{L^2}^2
	\]
	holds, where
	$C_B=R+\frac{R^2}{2r}  .$
\end{lem}
\begin{proof}
	Omitting the time $t$ for simplicity, we have
	\[\begin{split}
		r\|u\|_{H_0^1(\Omega)}^2 &=\int_\Omega r\|\nabla u\|_2^2\,\mathrm{d}x
		\\&\le\int_\Omega(\nabla u)^\mathrm{T}\pmb{a}(\nabla u)\,\mathrm{d}\pmb{x}
		\\&=A[u,u]-\int_\Omega(u\pmb{b}^\mathrm{T}(\nabla u)+cu^2)\,\mathrm{d}\pmb{x}
	.\end{split}\]
	The second term in the last equation above can be estimated as
	\[\begin{split}
		-\int_\Omega u\pmb{b}^\mathrm{T}(\nabla u)\,\mathrm{d}\pmb{x}
		&\le\int_\Omega\|b\|_2\cdot\|\nabla u\|_2\cdot|u|\,\mathrm{d}\pmb{x}
		\\&\le R\int_\Omega\|\nabla u\|_2\cdot|u|\,\mathrm{d}\pmb{x}
		\\&\le\frac r2\int_\Omega \|\nabla u\|_2^2\,\mathrm{d}\pmb{x}
		+\frac{R^2}{2r}\int_\Omega u^2\,\mathrm{d}\pmb{x}
		\\&=\frac r2\|u\|_{H_0^1(\Omega)}^2
		+\frac{R^2}{2r}\|u\|_{L^2}^2
	,\end{split}\]
	and the third term as
	\[\begin{split}
		-\int_\Omega cu^2\,\mathrm{d}\pmb{x}
		&\le\|c\|_{L^\infty}\int_\Omega u^2\,\mathrm{d}\pmb{x}
		\\&\le R\|u\|_{L^2}^2
	.\end{split}\]
	Combining the inequalities above, the lemma is proved. 
\end{proof}

\begin{lem}\label{lem:parabSolNorm}
	The weak solution to~\eqref{eq:parabEqOrig} satisfies
	\[
		\|\mathbf{u}\|_{L^2([0,T];H_0^1(\Omega))}
		\le C_1(\|\mathbf{f}\|_{L^2([0,T];H^{-1}(\Omega))}
		+\|g\|_{L^2})
	,\]
	where $C_1>0$ is a constant depending only on $r,R$ and $T$. 
\end{lem}
\begin{proof}
	Taking $v=\mathbf{u}(t)$ in~\eqref{eq:parabEqWeak}, we have
	\[
		\langle \mathbf{u}'(t),\mathbf{u}(t)\rangle+A[\mathbf{u}(t),\mathbf{u}(t);t]=\langle \mathbf{f}(t),\mathbf{u}(t)\rangle
	.\]
	This gives
	\begin{equation}\label{eq:parab_dtuL2}\begin{split}
		\frac{\mathrm{d}}{\mathrm{d}t}\bigl(\|\mathbf{u}(t)\|_{L^2}^2\bigr) &=2\langle\mathbf{u}'(t),\mathbf{u}(t)\rangle
		\\&=2\langle \mathbf{f}(t),\mathbf{u}(t)\rangle-2A[\mathbf{u}(t),\mathbf{u}(t);t]
		\\&\le 2\|\mathbf{f}(t)\|_{H^{-1}(\Omega)}\|\mathbf{u}(t)\|_{H_0^1(\Omega)}
		-2A[\mathbf{u}(t),\mathbf{u}(t);t]
		\\&\le\frac r2\|\mathbf{u}(t)\|_{H_0^1(\Omega)}^2
		+\frac2r\|\mathbf{f}(t)\|_{H^{-1}(\Omega)}^2
		-2A[\mathbf{u}(t),\mathbf{u}(t);t]
		\\&\le-\frac r2\|\mathbf{u}(t)\|_{H_0^1(\Omega)}^2
		+\frac2r\|\mathbf{f}(t)\|_{H^{-1}(\Omega)}^2
		+2C_B\|\mathbf{u}(t)\|_{L^2}^2
	.\end{split}\end{equation}
	The rest part of the proof has much in common with that of~\cite[Theorem~2, Section~7.1.2]{Evans2010PartialDE}, and we include the details for conveniece to the readers. 
	We denote $\alpha(t)=\|\mathbf{u}(t)\|_{L^2}^2$, $\beta(t)=\frac2r\|\mathbf{f}(t)\|_{H^{-1}(\Omega)}^2$. 
	Then~\eqref{eq:parab_dtuL2} implies $\alpha'(t)\le 2C_B\alpha(t)+\beta(t)$, and Gronwall's inequality yields the estimate
	\[\alpha(t)\le\exp(2C_Bt)\left(\alpha(0)+\int_0^t\beta(s)\,\mathrm{d}s\right) .\]
	As $\alpha(0)=\|\mathbf{u}(0)\|_{L^2}^2=\|g\|_{L^2}^2$, we have
	\[\alpha(t)\le\exp(2C_BT)\left(\|g\|_{L^2}^2+\frac2r\|\mathbf{f}\|_{L^2([0,T];H^{-1}(\Omega))}^2\right) \]
	for any $t\in[0,T]$. 
	We integrate~\eqref{eq:parab_dtuL2} from $0$ to $T$, and employ the inequality derived above to get
	\[
		\|\mathbf{u}\|_{L^2([0,T];H_0^1(\Omega))}^2
		\le C_1^2(\|\mathbf{f}\|_{L^2([0,T];H^{-1}(\Omega))}^2
		+\|g\|_{L^2}^2)
	,\]
	which directly implies the assertion of the lemma. 
\end{proof}
\begin{proof}[Proof of Proposition~\ref{prop:parabSLip}]
	We denote the bilinear forms corresponding to the two sets of coefficient fields as $A[\cdot,\cdot;t]$ and $\bar A[\cdot,\cdot;t]$, respectively. 
	Now for any $v\in H_0^1(\Omega)$, we let $\mathbf{w}=\mathbf{u}-\bar{\mathbf{u}}$, and substract the two equations
	\[\begin{split}
		\langle \mathbf{u}'(t),v\rangle+A[\mathbf{u}(t),v;t]&=\langle \mathbf{f}(t),v\rangle,
		\\\langle \bar {\mathbf{u}}'(t),v\rangle+\bar A[\bar {\mathbf{u}}(t),v;t]&=\langle \bar {\mathbf{f}}(t),v\rangle
	\end{split}\]
	to get
	\begin{equation}\label{eq:parab_w_eqn}
		\langle \mathbf{w}'(t),v\rangle+A[\mathbf{w}(t),v;t]=\langle \mathbf{f}_1(t),v\rangle
	,\end{equation}
	where $\mathbf{f}_1:[0,T]\to H^{-1}(\Omega)$ is given by
	\begin{equation}\label{eq:parab_f1def}
		\langle \mathbf{f}_1(t),v\rangle=\bar A[\bar {\mathbf{u}}(t),v;t]-A[\bar {\mathbf{u}}(t),v;t]
		+\langle(\mathbf{f}-\bar{\mathbf{f}})(t),v\rangle
	\end{equation}
	for any $v\in H_0^1(\Omega)$. 
	Now
	\[\begin{split}
		&\quad\bar A[u,v;t]-A[u,v;t]
		\\&=\int_\Omega((\nabla v)^\mathrm{T}(\pmb{a}-\pmb{\bar a})(\nabla u)
		\\&\quad+v(\pmb{\bar b}-\pmb{b})^\mathrm{T}(\nabla u)+(\bar c-c)uv)\,\mathrm{d}\pmb{x}
		\\&\le\int_\Omega\|\pmb{\bar a}-\pmb{a}\|_2\|\nabla u\|_2\|\nabla v\|_2\,\mathrm{d}\pmb{x}
		\\&\quad+\int_\Omega\bigl(\|\pmb{\bar b}-\pmb{b}\|_2\|\nabla u\|_2|v|+|\bar c-c|\cdot|u|\cdot|v|\bigr)\,\mathrm{d}\pmb{x}
		\\&\le\|\pmb{\bar a}-\pmb{a}\|_{L^\infty(\Omega_T,\ell^2(\R^{d\times d}))}\|u\|_{H_0^1(\Omega)}\|v\|_{H_0^1(\Omega)}
		\\&\quad+\|\pmb{\bar b}-\pmb{b}\|_{L^\infty(\Omega_T,\ell^2(\R^d))}\|u\|_{H_0^1(\Omega)}\|v\|_{L^2}
		\\&\quad+\|\bar c-c\|_{L^\infty(\Omega_T)}\|u\|_{L^2}\|v\|_{L^2}
		\\&\le \delta_{abc}\|u\|_{H_0^1(\Omega)}\|v\|_{H_0^1(\Omega)}
	,\end{split}\]
	where
	\begin{equation}\label{eq:parab_dabc_def}\begin{split}
		\delta_{abc}&=\|\pmb{\bar a}-\pmb{a}\|_{L^\infty(\Omega_T,\ell^2(\R^{d\times d}))}+C_\text{P}\|\pmb{\bar b}-\pmb{b}\|_{L^\infty(\Omega_T,\ell^2(\R^d))}
		\\&\quad+C_\text{P}^2\|\bar c-c\|_{L^\infty(\Omega_T)}
	,\end{split}\end{equation}
	and $C_\text{P}>0$ is the Poincar\'e constant of the domain $\Omega$. 
	By~\eqref{eq:parab_f1def}, we have
	\[\begin{split}
		\|\mathbf{f}_1(t)\|_{H^{-1}(\Omega)} &=\sup_{\|v\|_{H_0^1(\Omega)}=1}\langle\mathbf{f}_1(t),v\rangle
		\\&\le \delta_{abc}\|\bar{\mathbf{u}}(t)\|_{H_0^1(\Omega)}+\|\mathbf{f}(t)-\bar{\mathbf{f}}(t)\|_{H^{-1}(\Omega)}
	,\end{split}\]
	which holds for all $t\in[0,T]$. 
	These are non-negative functions with respect to $t$, and taking $L^2([0,T])$-norm gives
	\begin{equation}\label{eq:parab_f1est}
		\|\mathbf{f}_1\|_{L^2([0,T];H^{-1}(\Omega))}
		\le \delta_{abc}\|\bar{\mathbf{u}}\|_{L^2([0,T];H_0^1(\Omega))}+\|\mathbf{f}-\bar{\mathbf{f}}\|_{L^2([0,T];H^{-1}(\Omega))}
	.\end{equation}
	By Lemma~\ref{lem:parabSolNorm} and the condition
		$\|\bar{\mathbf{f}}\|_{L^2([0,T];H^{-1}(\Omega))}\le R$,
		$\|\bar g\|_{L^2}\le R$,
		we know
		\begin{equation}\label{eq:parab_wu_est}
		\|\bar{\mathbf{u}}\|_{L^2([0,T];H_0^1(\Omega))}\le 2C_1R
		.\end{equation}
	Applying Lemma~\ref{lem:parabSolNorm} to~\eqref{eq:parab_w_eqn} with initial condition $\mathbf{w}(0)=\mathbf{u}(0)-\bar{\mathbf{u}}(0)=g-\bar g$ would give
	\[\begin{split}
		&\quad\|\mathbf{u}-\bar{\mathbf{u}}\|_{L^2([0,T];H_0^1(\Omega))}
		\\&=\|\mathbf{w}\|_{L^2([0,T];H_0^1(\Omega))}
		\\&\le C_1(\|\mathbf{f}_1\|_{L^2([0,T];H^{-1}(\Omega))} +\|g-\bar g\|_{L^2})
		\\&\le C_1\delta_{abc}\|\bar{\mathbf{u}}\|_{L^2([0,T];H_0^1(\Omega))}+C_1\|\mathbf{f}-\bar{\mathbf{f}}\|_{L^2([0,T];H^{-1}(\Omega))}
		\\&\quad+C_1\|g-\bar g\|_{L^2}\hspace{6em}\text{(by~\eqref{eq:parab_f1est})}
		\\&\le 2RC_1^2\|\pmb{a}-\pmb{\bar a}\|_{L^\infty(\Omega_T,\ell^2(\R^{d\times d}))}
		+2RC_1^2C_\text{P}\|\pmb{b}-\pmb{\bar b}\|_{L^\infty(\Omega_T,\ell^2(\R^d))}
		\\&\quad+2RC_1^2C_\text{P}^2\|c-\bar c\|_{L^\infty(\Omega_T)}
		+C_1\|\mathbf{f}-\bar{\mathbf{f}}\|_{L^2([0,T];H^{-1}(\Omega))}
		\\&\quad+C_1\|g-\bar g\|_{L^2}\hspace{5em}\text{(by~\eqref{eq:parab_wu_est} and~\eqref{eq:parab_dabc_def})}
	.\end{split}\]
	The proof is completed by taking $C=\max(2RC_1^2,2RC_1^2C_\text{P}^2,C_1)$. 
\end{proof}
\begin{proof}[Proof of the assertion in Example~\ref{eg:parabEg1}]
	(1) 
	Given $\lambda\in\Lambda_K$, we denote $S_\Lambda(\lambda)=(\pmb{a},\pmb{b},c,\mathbf{f},g)$. 
	Since $\pmb{\phi}_k^b$, $k=1,\dots,K$ have disjoint essential supports, $|\lambda_k^b|\le\sqrt{\lambda_k^2}\le 1$, and $\pmb{\phi}_k^b\in B_R^d(\Omega_T)$, we have
	\begin{equation}\label{eq:parabEg1BnormEst1}\begin{split}
		\|\pmb{b}(t,\pmb{x})\|_2
		&=\left\lVert \sum_{k=1}^K\lambda_k^b\pmb{\phi}_k^b(t,\pmb{x})\right\rVert_2
		=\max_{1\le k\le K}\left\lVert\lambda_k^b\pmb{\phi}_k^b(t,\pmb{x})\right\rVert_2
		\\&\le\max_{1\le k\le K}\left\lVert\pmb{\phi}_k^b(t,\pmb{x})\right\rVert_2
		\le R
	\end{split}\end{equation}
	for a.e. $(t,x)\in\Omega_T$,
	and therefore
	\begin{equation}\label{eq:parabEg1BnormEst2}
		\|\pmb{b}\|_{L^\infty(\Omega_T;\ell^2(\R^d))}=
		\operatorname*{ess\,sup}_{t\in[0,T],\pmb{x}\in\Omega}\|\pmb{b}(t,\pmb{x})\|_2\le R
	.\end{equation}
	Similar arguments show $\|c\|_{L^\infty(\Omega_T)}\le R$, and 
	\[|\pmb{a}^{ii}(t,\pmb{x})-\tfrac{R+r}2|
		\le\max_{1\le k\le K}|\pmb{\phi}_k^a(t,\pmb{x})_i|
		\le\max_{1\le k\le K}\|\pmb{\phi}_k^a(t,\pmb{x})\|_2
	\le\tfrac{R-r}2 \]
	for $i=1,\dots,d$, a.e. $(t,\pmb{x})\in\Omega_T$. 
	Note that the latter conclusion would imply $\pmb{a}\in\Sigma_{r,R}^d(\Omega_T)$ since its off-diagonal entries are zero. 
	Furthermore, by Cauchy's inequality we have
	\begin{equation}\label{eq:parabEg1FnormEst}\begin{split}
		\|\mathbf{f}\|_{L^2([0,T];H^{-1}(\Omega))}^2
		&=\left\lVert \sum_{k=1}^K\lambda_k^\mathbf{f}\phi_k^\mathbf{f}(t)\right\rVert_{L^2([0,T];H^{-1}(\Omega))}^2
		\\&\le\biggl(\sum_{k=1}^K|\lambda_k^\mathbf{f}|\cdot\|\phi_k^\mathbf{f}(t)\|_{L^2([0,T];H^{-1}(\Omega))}\biggr)^2
		\\&\le\biggl(\sum_{k=1}^K(\lambda_k^\mathbf{f})^2\biggr)\biggl(\sum_{k=1}^K\left\lVert \phi_k^\mathbf{f}(t)\right\rVert_{L^2([0,T];H^{-1}(\Omega))}^2\biggr)
		\\&\le\biggl(\sum_{k=1}^K\lambda_k^2\biggr)R^2
		\\&\le R^2
	,\end{split}\end{equation}
	and similarly $\|g\|_{L^2(\Omega)}\le R$. 
	We conclude that $(\pmb{a},\pmb{b},c,\mathbf{f},g)\in\mathcal{A}_0$. 

	(2) 
	Let $\lambda,\bar\lambda\in\Lambda_K$, $S_\Lambda(\lambda)=(\pmb{a},\pmb{b},c,\mathbf{f},g)$ and $S_\Lambda(\bar\lambda)=(\pmb{\bar a},\pmb{\bar b},\bar c,\bar{\mathbf{f}},\bar g)$. 
	Similar to~\eqref{eq:parabEg1BnormEst1} and~\eqref{eq:parabEg1BnormEst2}, we have
	\begin{equation}\label{eq:parabDiff_b}\begin{split}
		\|\pmb{b}-\pmb{\bar b}\|_{L^\infty(\Omega_T;\ell^2(\R^d))}
		&=\max_{1\le k\le K}\left\lVert(\lambda_k^b-\bar\lambda_k^b)\pmb{\phi}_k^b(t,\pmb{x})\right\rVert_{L^\infty(\Omega_T;\ell^2(\R^d))}
		\\&\le\max_{1\le k\le K}\|\lambda-\bar\lambda\|\cdot\left\lVert\pmb{\phi}_k^b(t,\pmb{x})\right\rVert_{L^\infty(\Omega_T;\ell^2(\R^d))}
		\\&\le R\|\lambda-\bar\lambda\|
	,\end{split}\end{equation}
	and $\|c-\bar c\|_{L^2(\Omega_T)}\le R\|\lambda-\bar\lambda\|$. 
	The relation
	\[|\pmb{a}^{ii}(t,\pmb{x})-\pmb{\bar a}^{ii}(t,\pmb{x})|
		\le\tfrac{R-r}2\|\lambda-\bar\lambda\|
		\le R\|\lambda-\bar\lambda\|
	\]
	holds for $i=1,\dots,d$, a.e. $(t,\pmb{x})\in\Omega_T$, which implies $\|\pmb{a}-\pmb{\bar a}\|_{L^\infty(\Omega_T,\ell^2(\R^{d\times d}))}\le R\|\lambda-\bar\lambda\|$ since $\pmb{a}$ and $\pmb{\bar a}$ are diagonal matrices. 
	Furthermore, arguments analogous to~\eqref{eq:parabEg1FnormEst} would show
	$\|\mathbf{f}-\bar{\mathbf{f}}\|_{L^2([0,T];H^{-1}(\Omega))}^2\le R\|\lambda-\bar\lambda\|$ and $\|g-\bar g\|_{L^2(\Omega)}\le R\|\lambda-\bar\lambda\|$. 
	Combining the inequalities above, Proposition~\ref{prop:parabSLip} indicates
	\[
		\|\mathbf{u}-\bar{\mathbf{u}}\|_{L^2([0,T];H_0^1(\Omega))}
		\le 5CR\|\lambda-\bar\lambda\|
	,\]
	where $\mathbf{u}=S_0(S_\Lambda(\lambda))$ and $\bar{\mathbf{u}}=S_0(S_\Lambda(\bar\lambda))$. 
	We conclude that the mapping $S_0\circ S_\Lambda$ is Lipschitz continuous with Lipschitz constant $5CR$. 

	(3) According to the definition of the decoder width, it is easy to see that $d_{5K,1}^\text{rDeco}(\Lambda_K)=0$. 
	In fact, we may view $\Lambda_K$ as the closed unit ball of $Z=\R^{5K}$, and take $D=P_{\Lambda_K}:Z\to\Lambda_K$ to be the metric projection~\ref{eq:metricProjPK}. 
	Proposition~\ref{prop:decWidthLip} then implies $d_{5K,5CR}^\text{Deco}(S_0\circ S_\Lambda(\Lambda_K))=0$. 
\end{proof}
\begin{proof}[Proof of the assertion in Example~\ref{eg:parabEg2}]
	We have $d_{5n,1}^\text{rDeco}(\Lambda_w)\le\sqrt{w_n}$. 
	In fact, given $n>0$, we may take $D_1:Z=\R^{5n}\to\{\lambda\}$ to be an inclusion that fills the entries $\lambda_k^a,\dots,\lambda_k^g$ for $k=1,\dots,n$. 
	Denote $Z_w=D_1^{-1}(\Lambda_w)$ to be the preimage of $\Lambda_w$. 
	It is easy to check that $Z_w$ is a convex subset of $Z$. 
	$Z_w\subseteq Z_B$ holds, since for any $z\in Z_w$, $\lambda=D_1(z)\in\Lambda_w$ satisfies $\lambda_k^2=0$ for $k>n$, and we have
	\[
		\|z\|_2^2=\sum_{k=1}^{n}\lambda_k^2
		\le\sum_{k=1}^{n}\frac{\lambda_k^2}{w_k}
		=\sum_{k=1}^\infty\frac{\lambda_k^2}{w_k}
		\le 1
	.\]
	For any $\lambda\in\Lambda_w$, we can find $z\in Z_w=P_{Z_w}(Z_B)$ satisfying
	\[
		\|\lambda-D_1(z)\|^2=\sum_{k>n}\lambda_k^2
		\le w_n\sum_{k>n}\frac{\lambda_k^2}{w_k}
		\le w_n\sum_{k=1}^\infty\frac{\lambda_k^2}{w_k}
		\le w_n
	.\]
	Then $D=D_1\circ P_{Z_w}:Z\to\Lambda_w$ can be chosen as the specific decoder mapping, 
	and $d_{5n,1}^\text{rDeco}(\Lambda_w)\le\sqrt{w_n}$ follows from the definition of the restricted decoder width. 

	Similar to Example~\ref{eg:parabEg1}, 
	it can be shown that the mapping $S_0\circ S_\Lambda:\Lambda_w\to L^2([0,T];H_0^1(\Omega))$ is $5CR$-Lipschitz continuous. 
	Proposition~\ref{prop:decWidthLip} then implies
	\[d_{5n,5CR}^\text{Deco}(S_0\circ S_\Lambda(\Lambda_w))\le 5CR\sqrt{w_n} .\qedhere\] 
\end{proof}

\subsection{Elliptic Equations on a Variable Domain}\label{sec:ellipVDproof}
\subsubsection{Preparation: Solutions on Two Domains}\label{sec:ellip2domain}
We shall first focus on the simpler case where only two possible domains $\Omega$ and $\Omega_2$ are considered, and estimate the distance between the corresponding solutions $u,u_2$. 
This analysis forms the basis of the subsequent proofs in Section~\ref{sec:ellipVDproof}, where we will investigate the more general case involving a whole class of domains.

Assume $\Omega\subseteq\Omega^\text{m}$ is a fixed subdomain, $\beta$ is a $C^1$-diffeomorphism from $\bar\Omega^\text{m}$ onto itself. 
Denote $\Omega_2=\beta(\Omega)$, 
\[
	\delta_\beta=\|\beta-\mathrm{id}\|_{C^1(\bar\Omega^\text{m};\R^d)}
	=\max_{\pmb{x}\in\bar\Omega^\text{m}}\max\left(\|\beta(\pmb{x})-\pmb{x}\|_2,\left\|\frac{\partial \beta(\pmb{x})}{\partial \pmb{x}}-I_d\right\|_2\right)
,\]
where $\mathrm{id}$ is the identity mapping from $\Omega^\text{m}$ to itself. 
The weak solution to~\eqref{eq:ellipVarDom} is denoted as $u\in H^1_0(\Omega^\text{m}\setminus\partial\Omega)$, 
and the solution with the inner region $\Omega$ replaced by $\Omega_2$ is denoted as $u_2\in H^1_0(\Omega^\text{m}\setminus\partial\Omega_2)$. 
We state the main theorem for this section (Section~\ref{sec:ellip2domain}) as follows: 
\begin{thm}\label{thm:vardomEllip}
	Under Assumption~\ref{ass:vardomEllip}, we have
	$
		\|u-u_2\|_{L^2(\Omega^\text{m})}\le C\delta_\beta/2
	$
	provided $\delta_\beta\le\delta_0\triangleq\min(1-(8/9)^{2/d},1/9)$, 
	where the constant $C$ is the same as in Proposition~\ref{prop:eVDtransLip}
\end{thm}
We let $\tilde u(\pmb{x})=u(\beta^{-1}(\pmb{x}))$, and it is easy to verify that $\tilde u\in H^1_0(\Omega^\text{m}\setminus\partial\Omega_2)$ holds. 
To prove Theorem~\ref{thm:vardomEllip}, we shall estimate $\|u-\tilde u\|_{L^2(\Omega^\text{m})}$ and $\|\tilde u-u_2\|_{L^2(\Omega^\text{m})}$ seperately. 
For the former term, we construct a continuous transition between $u$ and $\tilde u$ in $L^2(\Omega^\text{m})$ by using an interpolation between $\beta$ and the identity mapping on $\bar\Omega^\text{m}$ (see Proposition~\ref{prop:dv_beta_L2} and Corollary~\ref{cor:du_tildeu}). 
For the latter term, we derive the PDE that $\tilde u$ satisfies (Lemma~\ref{lem:eqn4tildeu}), 
and then the upper bound can be given
as in the fixed domain case (Proposition~\ref{prop:ellpLip}). 

The detailed proof will be presented in the following part of this section. 
We assume Assumption~\ref{ass:vardomEllip} 
holds unless otherwise specified. 

\begin{lem}\label{lem:linalgDet}
	Let $Q\in\R^{d\times d}$ be a matrix. Then $\det(I+Q)\le(1+\alpha)^d$, where $\alpha=\|Q\|_2$. 
	If $\alpha<1$, we further have $\det(I+Q)\ge(1-\alpha)^d$. 
\end{lem}
\begin{proof}
	If $Q$ is symmetric, we may let $Q=P^\mathrm{T}\Lambda P$, where $P$ is orthogonal and $\Lambda$ is diagonal. 
	Then $\alpha=\|P^\mathrm{T}\Lambda P\|_2=\|\Lambda\|_2$. 
	By the definition of $\ell^2$-norm, we have $\|\Lambda\|_2=\max_i|\Lambda_{ii}|$ since it is diagonal, and
	\[\begin{split}
		\det(I+Q)&=\det(I+P^\mathrm{T}\Lambda P)=\det(P^\mathrm{T}(I+\Lambda)P)
		\\&= \det(I+\Lambda)=\prod_{i=1}^d(1+\Lambda_{ii})
		\\&\le \prod_{i=1}^d|1+\Lambda_{ii}| \le(\max_i|1+\Lambda_{ii}|)^d
		\\&\le (1+\max_i|\Lambda_{ii}|)^d=(1+\|\Lambda\|_2)^d
		\\&= (1+\alpha)^d
	.\end{split}\]
	For non-symmetric $Q$, 
	\[
		\det(I+Q)^2=\det(I+Q)\det(I+Q^\mathrm{T})=\det(I+Q+Q^\mathrm{T}+QQ^\mathrm{T})
	.\]
	The matrix $Q_1=Q+Q^\mathrm{T}+QQ^\mathrm{T}$ is symmetric, and
	\[ \|Q_1\|_2\le\|Q\|_2+\|Q^\mathrm{T}\|_2+\|QQ^\mathrm{T}\|_2\le 2\|Q\|_2+\|Q\|_2^2=2\alpha+\alpha^2 .\]
	Then $\det(I+Q_1)\le(1+2\alpha+\alpha^2)^d=(1+\alpha)^{2d}$ by what we have proved. 
	This shows $\det(I+Q)\le(1+\alpha)^d$. 

	When $\alpha<1$, it is easy to verify that the sum
	\[ I-Q+Q^2-Q^3+\cdots=(I+Q)^{-1} \]
	is convergent. We let $Q_2=(I+Q)^{-1}-I$, and
	\[\begin{split}
		\|Q_2\|_2&=\left\lVert -Q+Q^2-Q^3+\cdots\right\rVert_2
		\\&\le\|Q\|_2+\|Q^2\|_2+\|Q^3\|_2+\cdots
		\\&\le\alpha+\alpha^2+\alpha^3+\cdots
		=\frac{\alpha}{1-\alpha}
	,\end{split}\]
	\[ \det(I+Q_2)\le(1+\|Q_2\|_2)^d\le\left(1+\frac\alpha{1-\alpha}\right)^d=\frac1{(1-\alpha)^d} .\]
	The conclusion follows from $\det(I+Q)=1/\det(I+Q_2)$. 
\end{proof}

\begin{cor}\label{cor:detJ}
	Let $J(\pmb{x})=\frac{\partial \beta(\pmb{x})}{\partial \pmb{x}}$ be the Jacobian matrix of $\beta$ at $\pmb{x}\in\Omega^\text{m}$. 
	When $\delta_\beta\le\delta_0$, we have
	\[(1-\delta_\beta)^d\le\det J(\pmb{x})\le(1+\delta_\beta)^d .\]
\end{cor}
\begin{proof}
	Note that $\|J(\pmb{x})-I_d\|_2=\left\|\frac{\partial \beta(\pmb{x})}{\partial \pmb{x}}-I_d\right\|_2\le\delta_\beta$. 
	By Lemma~\ref{lem:linalgDet}, we have
	\[ \det J(\pmb{x})=\det(I_d+(J(\pmb{x})-I_d))\le(1+\|J(\pmb{x})-I_d\|_2)^d\le(1+\delta_\beta)^d
	.\]
	Since $\delta_\beta\le\delta_0<1$, the other side of the inequality can be derived from a similar argument. 
\end{proof}

\begin{prop}\label{prop:est_tildeuH01}
	We have $\|u\|_{H_0^1(\Omega^\text{m})}\le C_\text{P}R/r$ under Assumption~\ref{ass:vardomEllip}, where $C_\text{P}$ is the Poincar\'e constant of $\Omega^\text{m}$. 
	When $\delta_\beta\le\delta_0$, we further have $\|\tilde u\|_{H_0^1(\Omega^\text{m})}\le 2C_\text{P}R/r$. 
\end{prop}
\begin{proof}
	The proof of $\|u\|_{H_0^1(\Omega^\text{m})}\le C_\text{P}R/r$ is the same as Lemma~\ref{lem:ellpH01u} for the case of a fixed domain, in which we may take $c(\pmb{x})\equiv 0$. 
	Now assume $\delta_\beta\le\delta_0$. For $\pmb{x}\in\Omega^\text{m}$, we have $u(\pmb{x})=\tilde u(\beta(\pmb{x}))$, and $\nabla u(\pmb{x})=J(\pmb{x})^\mathrm{T}\nabla\tilde u(\beta(\pmb{x}))$. This gives
	\[\begin{split}
		\|\nabla u(\pmb{x})\|_2&\ge\|\nabla\tilde u(\beta(\pmb{x}))\|_2-\|(I_d-J(\pmb{x}))^\mathrm{T}\nabla\tilde u(\beta(\pmb{x}))\|_2
		\\&\ge\|\nabla\tilde u(\beta(\pmb{x}))\|_2-\|J(\pmb{x})-I_d\|_2\|\nabla\tilde u(\beta(\pmb{x}))\|_2
		\\&\ge\|\nabla\tilde u(\beta(\pmb{x}))\|_2-\delta_\beta\|\nabla\tilde u(\beta(\pmb{x}))\|_2
		\\&=(1-\delta_\beta)\|\nabla\tilde u(\beta(\pmb{x}))\|_2
	,\end{split}\]
	and
	\[\begin{split}
		(1-\delta_\beta)^{-2}\|u\|_{H_0^1(\Omega^\text{m})}^2
		&=(1-\delta_\beta)^{-2}\int_{\Omega^\text{m}}\|\nabla u(\pmb{x})\|_2^2\,\mathrm{d}\pmb{x}
		\\&\ge\int_{\Omega^\text{m}}\|\nabla\tilde u(\beta(\pmb{x}))\|_2^2\,\mathrm{d}\pmb{x}
		\\&=\int_{\Omega^\text{m}}\|\nabla\tilde u(\pmb{y})\|_2^2\det J(\beta^{-1}(\pmb{y}))^{-1}\,\mathrm{d}\pmb{y}
		\\&\ge(1+\delta_\beta)^{-d}\int_{\Omega^\text{m}}\|\nabla\tilde u(\pmb{y})\|_2^2\,\mathrm{d}\pmb{y}
		\\&=(1+\delta_\beta)^{-d}\|\tilde u\|_{H_0^1(\Omega^\text{m})}^2
	\end{split}\]
	by Corollary~\ref{cor:detJ}. Then
	\[\begin{split}
		\|\tilde u\|_{H_0^1(\Omega^\text{m})}
		&\le(1+\delta_\beta)^{d/2}(1-\delta_\beta)^{-1}\|u\|_{H_0^1(\Omega^\text{m})}
		\\&\le \frac98\cdot\left(1-\frac19\right)^{-1}\frac{C_\text{P}R}r
		\\&\le \frac{2C_\text{P}R}r
	.\end{split}\]
	Here we used the fact that $\delta_\beta\le\delta_0\le 1-(8/9)^{2/d}$ implies $(1+\delta_\beta)^{d/2}\le(1-\delta_\beta)^{-d/2}\le 9/8$. 
\end{proof}

\begin{prop}\label{prop:dv_beta_L2}
	Assume $\delta_\beta\le\delta_0$. Then for $v\in H_0^1(\Omega^\text{m})$, we have
	\[\|v\circ\beta-v\|_{L^2(\Omega^\text{m})}\le\|v\|_{H_0^1(\Omega^\text{m})}(1-\delta_\beta)^{-d/2}\delta_\beta
		\le 2\|v\|_{H_0^1(\Omega^\text{m})}\delta_\beta
	.\]
\end{prop}
\begin{proof}
	Since $C_0^1(\Omega^\text{m})$ is dense in $H_0^1(\Omega^\text{m})$, we may assume without loss of generality that $v\in C_0^1(\Omega^\text{m})$. 
	Define $\beta_t(\pmb{x})=\pmb{x}+t(\beta(\pmb{x})-\pmb{x})$ for $t\in[0,1]$. 
	Then $\beta_0=\mathrm{id}$, $\beta_1=\beta$, $\beta_t(\pmb{x})\in\Omega^\text{m}$ by the convexity of $\Omega^\text{m}$, and
	\[\begin{split}
		\frac{\mathrm{d}}{\mathrm{d}t}v(\beta_t(\pmb{x}))&=\nabla v(\beta_t(\pmb{x}))\cdot(\beta(\pmb{x})-\pmb{x}),
		\\\left\lvert \frac{\mathrm{d}}{\mathrm{d}t}v(\beta_t(\pmb{x}))\right\rvert
		&\le|\nabla v(\beta_t(\pmb{x}))|\cdot\|\beta-\mathrm{id}\|_{L^\infty(\Omega^\text{m};\R^d)}
		\\&\le\delta_\beta|\nabla v(\beta_t(\pmb{x}))|
		,\\(v(\beta(\pmb{x}))-v(\pmb{x}))^2&=\left(\int_0^1\frac{\mathrm{d}}{\mathrm{d}t}v(\beta_t(\pmb{x}))\,\mathrm{d}t\right)^2
		\\&\le\int_0^1\left(\frac{\mathrm{d}}{\mathrm{d}t}v(\beta_t(\pmb{x}))\right)^2\,\mathrm{d}t\cdot\int_0^11^2\,\mathrm{d}t
		\\&\le\delta_\beta^2\int_0^1|\nabla v(\beta_t(\pmb{x}))|^2\,\mathrm{d}t
	.\end{split}\]
	As $\|\frac{\partial \beta(\pmb{x})}{\partial \pmb{x}}-I_d\|_2\le\delta_\beta$ for any $\pmb{x}\in\Omega^\text{m}$, the mapping $\pmb{x}\mapsto\beta(\pmb{x})-\pmb{x}$ is $\delta_\beta$-Lipschitz continuous, and $\beta_t(\pmb{x})$ is an injection. 
	In fact, if $\beta_t(\pmb{x}_1)=\beta_t(\pmb{x}_2)$ for some $t\in(0,1)$, then $\pmb{x}_2-\pmb{x}_1=t(\beta(\pmb{x}_1)-\pmb{x}_1)-t(\beta(\pmb{x}_2)-\pmb{x}_2)$, and $|\pmb{x}_1-\pmb{x}_2|\le t\delta_\beta|\pmb{x}_1-\pmb{x}_2|$, thereby $|\pmb{x}_1-\pmb{x}_2|=0$ since $t\delta_\beta<\delta_\beta\le 1$. 
	\par Now for any $\pmb{x}\in\Omega^\text{m}$, $\|\frac{\partial \beta_t(\pmb{x})}{\partial \pmb{x}}-I_d\|_2=\|t(\frac{\partial \beta(\pmb{x})}{\partial \pmb{x}}-I_d)\|_2\le t\delta_\beta$, and Lemma~\ref{lem:linalgDet} would imply $\det\frac{\partial \beta_t(\pmb{x})}{\partial \pmb{x}}\ge(1-t\delta_\beta)^d\ge(1-\delta_\beta)^d$. 
	We have
	\[\begin{split}
		\int_{\Omega^\text{m}}|\nabla v(\beta_t(\pmb{x}))|^2\,\mathrm{d}\pmb{x}
		&=\int_{\beta_t(\Omega^\text{m})}|\nabla v(\pmb{y})|^2\left(\det\frac{\partial \beta_t}{\partial \pmb{x}}\biggm|_{\pmb{x}=\beta_t^{-1}(\pmb{y})}\right)^{-1}\,\mathrm{d}\pmb{y}
		\\&\le\int_{\Omega^\text{m}}|\nabla v(\pmb{y})|^2(1-\delta_\beta)^{-d}\,\mathrm{d}\pmb{y}
		\\&=(1-\delta_\beta)^{-d}\|v\|_{H_0^1(\Omega^\text{m})}^2
		,\\\|v\circ\beta-v\|_{L^2(\Omega^\text{m})}^2&=\int_{\Omega^\text{m}}(v(\beta(\pmb{x}))-v(\pmb{x}))^2\,\mathrm{d}\pmb{x}
		\\&\le\int_{\Omega^\text{m}}\delta_\beta^2\int_0^1|\nabla v(\beta_t(\pmb{x}))|^2\,\mathrm{d}t\,\mathrm{d}\pmb{x}
		\\&\le\delta_\beta^2\int_0^1(1-\delta_\beta)^{-d}\|v\|_{H_0^1(\Omega^\text{m})}^2\,\mathrm{d}t
		\\&=\delta_\beta^2(1-\delta_\beta)^{-d}\|v\|_{H_0^1(\Omega^\text{m})}^2
	.\end{split}\]
	We use $(1-\delta_\beta)^{-d/2}\le 9/8< 2$ to conclude the proof. 
\end{proof}

Taking $v=\tilde u$ in Proposition~\ref{prop:dv_beta_L2}, a direct combination with Proposition~\ref{prop:est_tildeuH01} gives
\begin{cor}\label{cor:du_tildeu}
	Assuming $\delta_\beta\le\delta_0$ and Assumption~\ref{ass:vardomEllip} holds. Then $\|u-\tilde u\|_{L^2(\Omega^\text{m})}\le\frac4rC_\text{P}R\delta_\beta$. 
\end{cor}

\begin{lem}\label{lem:eqn4tildeu}
	$\tilde u\in H^1_0(\Omega^\text{m}\setminus\partial\Omega_2)$ is the weak solution to $\tilde L\tilde u=\tilde f$ in $\Omega^\text{m}\setminus\partial\Omega_2$, where
	\[\begin{split}
		\tilde Lu(\pmb{y})&=-\nabla_y\cdot(\pmb{\tilde a}(\pmb{y})\nabla_yu(\pmb{y}))
		,\\\pmb{\tilde a}(\pmb{y})&=\pmb{\hat a}(\beta^{-1}(\pmb{y}))
		,\\\pmb{\hat a}(\pmb{x})&=\frac{J(\pmb{x})\pmb{a}(\pmb{x})J(\pmb{x})^\mathrm{T}}{\det J(\pmb{x})}
		,\\\tilde f(\pmb{y})&=\frac{f(\beta^{-1}(\pmb{y}))}{\det J(\beta^{-1}(\pmb{y}))}
	\end{split}\]
	for all $\pmb{y}\in\Omega^\text{m}$, $\pmb{x}=\beta^{-1}(\pmb{y})$. 
\end{lem}
\begin{proof}
	For any $\tilde v\in H^1_0(\Omega^\text{m}\setminus\partial\Omega_2)$, we let $v(\pmb{x})=\tilde v(\beta(\pmb{x}))$, and then 
	\[\nabla_xv(\pmb{x})=J(\pmb{x})^\mathrm{T}\nabla_y\tilde v(\beta(\pmb{x})).\]
	Since $\beta(\pmb{x})$ is a $C^1$-diffeomorphism from $\bar\Omega^\text{m}$ to itself, $J(\pmb{x})$ is continuous on $\bar\Omega^\text{m}$, and then $\|J(\pmb{x})\|_2$ must be bounded. 
	This implies $v\in H^1_0(\Omega^\text{m}\setminus\partial\Omega)$. 
	We have
	\[\begin{split}
		\int_{\Omega^\text{m}}f(\pmb{x})v(\pmb{x})\,\mathrm{d}\pmb{x}
		&=\int_{\Omega^\text{m}}(\nabla_xu(\pmb{x}))^\mathrm{T}\pmb{a}(\pmb{x})\nabla_xv(\pmb{x})\,\mathrm{d}\pmb{x}
		\\&=\int_{\Omega^\text{m}}(J(\pmb{x})^\mathrm{T}\nabla_y\tilde u(\beta(\pmb{x})))^\mathrm{T}\pmb{a}(\pmb{x})J(\pmb{x})^\mathrm{T}\nabla_y\tilde v(\beta(\pmb{x}))\,\mathrm{d}\pmb{x}
		\\&=\int_{\Omega^\text{m}}(\nabla_y\tilde u(\beta(\pmb{x})))^\mathrm{T}\pmb{\hat a}(\pmb{x})\nabla_y\tilde v(\beta(\pmb{x}))\det J(\pmb{x})\,\mathrm{d}\pmb{x}
		\\&=\int_{\Omega^\text{m}}(\nabla_y\tilde u(\pmb{y}))^\mathrm{T}\pmb{\hat a}(\beta^{-1}(\pmb{y}))\nabla_y\tilde v(\pmb{y})\,\mathrm{d}\pmb{y}
		\\&=\int_{\Omega^\text{m}}(\nabla_y\tilde u(\pmb{y}))^\mathrm{T}\pmb{\tilde a}(\pmb{y})\nabla_y\tilde v(\pmb{y})\,\mathrm{d}\pmb{y}
	.\end{split}\]
	Furthermore, 
	\begin{equation}\label{eq:vardomEllip_ftilde}\begin{split}
		\int_{\Omega^\text{m}}f(\pmb{x})v(\pmb{x})\,\mathrm{d}\pmb{x}
		&=\int_{\Omega^\text{m}}\tilde f(\beta(\pmb{x}))\tilde v(\beta(\pmb{x}))\det J(\pmb{x})\,\mathrm{d}\pmb{x}
		\\&=\int_{\Omega^\text{m}}\tilde f(\pmb{y})\tilde v(\pmb{y})\,\mathrm{d}\pmb{y}
	.\end{split}\end{equation}
	Since these equalities hold for any $\tilde v\in H^1_0(\Omega^\text{m}\setminus\partial\Omega_2)$, the conclusion follows from the definition of a weak solution. 
\end{proof}

\begin{prop}\label{prop:df_tildef}
	If $\delta_\beta\le\delta_0$, 
	we have $\|\tilde f-f\|_{H^{-1}(\Omega^\text{m})}\le R\delta_\beta(1-\delta_\beta)^{-d/2}\le 2R\delta_\beta$. 
\end{prop}
\begin{proof}
	
	Assume $v\in H^1_0(\Omega^\text{m})$. 
	Similar to~\eqref{eq:vardomEllip_ftilde}, 
	\[\langle\tilde f,v\rangle\triangleq
		\int_{\Omega^\text{m}}\tilde f(\pmb{x})v(\pmb{x})\,\mathrm{d}\pmb{x}
		=\int_{\Omega^\text{m}}f(\pmb{x})v(\beta(\pmb{x}))\,\mathrm{d}\pmb{x}
	,\]
	and therefore
	\[\begin{split}
		\langle\tilde f-f,v\rangle&=\langle f,v\circ\beta-v\rangle
		\\&\le\|f\|_{L^2(\Omega^\text{m})}\|v\circ\beta-v\|_{L^2(\Omega^\text{m})}
		\\&\le R\cdot \|v\|_{H_0^1(\Omega^\text{m})}(1-\delta_\beta)^{-d/2}\delta_\beta
	,\end{split}\]
	where the last inequality makes use of Assumption~\ref{ass:vardomEllip}(3) and Proposition~\ref{prop:dv_beta_L2}. 
	The conclusion follows immediately by the arbitrariness of $v$. 
\end{proof}

\begin{prop}\label{prop:estiJx}
	Assume $\delta_\beta<1$. 
	For all $\pmb{x}\in\Omega^\text{m}$, we have
	\[\left\lVert \frac{J(\pmb{x})}{\sqrt{\det J(\pmb{x})}}-I\right\rVert_2\le\delta_\beta(1-\delta_\beta)^{-d/2}+(1-\delta_\beta)^{-d/2}-1
	.\] 
	Furthermore, if $\delta_\beta\le\delta_0$, then
	\[\left\lVert \frac{J(\pmb{x})}{\sqrt{\det J(\pmb{x})}}-I\right\rVert_2\le C_J\delta_\beta\le\frac14
	,\] 
	where $C_J=\bigl(9+\delta_0^{-1}\bigr)/8$. 
\end{prop}
\begin{proof}
	Making use of Corollary~\ref{cor:detJ}
	and the fact that the function $\psi(s)=|\frac1{\sqrt s}-1|$ achieves its maximum at endpoints, 
	we have
	\[\begin{split}
		\left\lVert \frac{J(\pmb{x})}{\sqrt{\det J(\pmb{x})}}-I\right\rVert_2
		&\le\left\lVert \frac{J(\pmb{x})-I}{\sqrt{\det J(\pmb{x})}}\right\rVert_2
		+\left\lVert\biggl(\frac1{\sqrt{\det J(\pmb{x})}}-1\biggr)I\right\rVert_2
		\\&=\frac{\|J(\pmb{x})-I\|_2}{\sqrt{\det J(\pmb{x})}}
		+\left|\frac1{\sqrt{\det J(\pmb{x})}}-1\right|
		\\&\le\delta_\beta(1-\delta_\beta)^{-d/2}+\max((1-\delta_\beta)^{-d/2}-1, 1-(1+\delta_\beta)^{-d/2})
		\\&=\delta_\beta(1-\delta_\beta)^{-d/2}+(1-\delta_\beta)^{-d/2}-1
	.\end{split}\]
	When $\delta_\beta\le\delta_0$, 
	the condition $\delta_\beta\le 1-(8/9)^{2/d}$ implies $(1-\delta_\beta)^{-d/2}\le 9/8$. 
	The convexity of the function $\phi(t)=(1-t)^{-d/2}$ on $0\le t<1$ ensures $\frac{\phi(\delta_\beta)-\phi(0)}{\delta_\beta}\le\frac{\phi(\delta_0)-\phi(0)}{\delta_0}$, hence
	\[ (1-\delta_\beta)^{-d/2}-1\le((1-\delta_0)^{-d/2}-1)\frac{\delta_\beta}{\delta_0}\le\left(\frac98-1\right)\frac{\delta_\beta}{\delta_0}=\frac{\delta_\beta}{8\delta_0}
	.\]
	This gives
	\[\delta_\beta(1-\delta_\beta)^{-d/2}+(1-\delta_\beta)^{-d/2}-1
		\le\frac98\delta_\beta+\frac{\delta_\beta}{8\delta_0}=C_J\delta_\beta
	.\] 
	Since $\delta_\beta\le\frac19$ holds as well, we further have
	\[ \frac98\delta_\beta+\frac18\cdot\frac{\delta_\beta}{\delta_0}\le\frac98\cdot\frac19+\frac18\cdot 1=\frac14
	.\] 
\end{proof}

\begin{prop}
	If $\delta_\beta\le\delta_0$, then $\pmb{\tilde a}\in\Sigma_{r/2,2R}^d(\Omega^\text{m})$. 
\end{prop}
\begin{proof}
	It suffices to prove $\pmb{\hat a}\in\Sigma_{r/2,2R}^d(\Omega^\text{m})$. 
	For $\pmb{x}\in\Omega^\text{m}$, we denote for simplicity $B=\frac{J(\pmb{x})}{\sqrt{\det J(\pmb{x})}}$. 
	Proposition~\ref{prop:estiJx} shows $\|B-I\|_2\le 1/4$. 
	Then
	\[\|B\xi\|_2\ge\|\xi\|_2-\|(B-I)\xi\|_2\ge\|\xi\|_2-\|B-I\|_2\|\xi\|_2\ge\frac34\|\xi\|_2,\] 
	and $\|B\xi\|_2^2\ge\frac9{16}\|\xi\|_2^2\ge\frac12\|\xi\|_2^2$. 
	Likewise, we have $\|B\xi\|_2^2\le(1+\frac14)^2\|\xi\|_2^2\le 2\|\xi\|_2^2$. 
	Since $a\in\Sigma_{r,R}^d(\Omega^\text{m})$, for any $\xi\in\R^d$,
	\begin{equation}\label{eq:vardomEllipUAE}
		\xi^\mathrm{T}\pmb{\hat a}(\pmb{x})\xi=(B\xi)^\mathrm{T}\pmb{a}(\pmb{x})B\xi
		\in[r\|B\xi\|_2^2,R\|B\xi\|_2^2]
		\subseteq\left[\frac r2\|\xi\|_2^2,2R\|\xi\|_2^2\right]
	\end{equation}
	holds, and the proposition is proved. 
\end{proof}

\begin{prop}\label{prop:da_tildea}
	If $\delta_\beta\le\delta_0$, we have
	\[ \|\pmb{a}-\pmb{\tilde a}\|_{L^\infty(\Omega^\text{m};\ell^2(\R^{d\times d}))}\le R(3C_J+1)\delta_\beta
	.\]
\end{prop}
\begin{proof}
	For $\pmb{x}\in\Omega^\text{m}$, by Assumption~\ref{ass:vardomEllip}(1), 
	we have
	\[\begin{aligned}
		\|\pmb{a}(\pmb{x})-\pmb{a}(\beta^{-1}(\pmb{x}))\|_2
		&\le R\|\pmb{x}-\beta^{-1}(\pmb{x})\|_2
		\\&= R\|\beta(\beta^{-1}(\pmb{x}))-\beta^{-1}(\pmb{x})\|_2
		\\&\le R\max_{\pmb{y}\in\Omega^\text{m}}\|\beta(\pmb{y})-\pmb{y}\|_2
		\\&\le R\delta_\beta
	.\end{aligned}\]
	Denote for simplicity $B=\frac{J(\beta^{-1}(\pmb{x}))}{\sqrt{\det J(\beta^{-1}(\pmb{x}))}}$, and $A=\pmb{a}(\beta^{-1}(\pmb{x}))$. 
	Then $\pmb{\tilde a}(\pmb{x})=B^\mathrm{T}AB$, and
	\[\begin{split}
		\|\pmb{\tilde a}(\pmb{x})-\pmb{a}(\beta^{-1}(\pmb{x}))\|_2&=\|B^\mathrm{T}AB-A\|_2
		\\&\le\|B^\mathrm{T}AB-AB\|_2+\|AB-A\|_2
		\\&=\|(B-I)^\mathrm{T}AB\|_2+\|A(B-I)\|_2
		\\&\le\|B-I\|_2\|A\|_2\bigl(\|B\|_2+1\bigr)
	.\end{split}\]
	By Proposition~\ref{prop:estiJx} we have $\|B-I\|_2\le C_J\delta_\beta$, and
	\[ \|B\|_2\le\|B-I\|_2+\|I\|_2\le\frac14+1\le 2
	.\]
	Since $\pmb{a}\in\Sigma_{r,R}^d(\Omega^\text{m})$, the matrix $A$ is symmetric and positive definite, which indicates
	\[\|A\|_2=\lambda_{\max}(A)=\max_{\xi\in\R^d,\|\xi\|_2=1}\xi^\mathrm{T}A\xi\le R\cdot 1^2
	.\]
	The conclusion follows by
	\[\begin{split}
		\|\pmb{a}(\pmb{x})-\pmb{\tilde a}(\pmb{x})\|_2&\le\|\pmb{a}(\pmb{x})-\pmb{a}(\beta^{-1}(\pmb{x}))\|_2+\|\pmb{\tilde a}(\pmb{x})-\pmb{a}(\beta^{-1}(\pmb{x}))\|_2
		\\&\le R\delta_\beta+C_J\delta_\beta\cdot R\cdot(2+1)
		\\&=R(3C_J+1)\delta_\beta
	.\qedhere\end{split}\]
\end{proof}

We are now ready to prove the main theorem of Section~\ref{sec:ellip2domain}. 
\begin{proof}[Proof of Theorem~\ref{thm:vardomEllip}]
	We first estimate $\|u_2-\tilde u\|_{H_0^1(\Omega^\text{m}\setminus\partial\Omega_2)}$. 
	According to Lemma~\ref{lem:eqn4tildeu} and the definition of $u_2$, we have
	\[\begin{split}
		\int_{\Omega^\text{m}}(\nabla\tilde u(\pmb{x}))^\mathrm{T}\pmb{\tilde a}(\pmb{x})\nabla v(\pmb{x})\,\mathrm{d}\pmb{x}
		&=\int_{\Omega^\text{m}}\tilde f(\pmb{x})v(\pmb{x})\,\mathrm{d}\pmb{x}
		,\\\int_{\Omega^\text{m}}(\nabla u_2(\pmb{x}))^\mathrm{T}\pmb{a}(\pmb{x})\nabla v(\pmb{x})\,\mathrm{d}\pmb{x}
		&=\int_{\Omega^\text{m}}f(\pmb{x})v(\pmb{x})\,\mathrm{d}\pmb{x}
	\end{split}\]
	for any $v\in H_0^1(\Omega^\text{m}\setminus\partial\Omega_2)$, 
	and subtraction gives
	\[\begin{split}
		&\quad\int_{\Omega^\text{m}}(\nabla u_2(\pmb{x})-\nabla\tilde u(\pmb{x}))^\mathrm{T}\pmb{a}(\pmb{x})\nabla v(\pmb{x})\,\mathrm{d}\pmb{x}
		\\&=\int_{\Omega^\text{m}}(\nabla\tilde u(\pmb{x}))^\mathrm{T}(\pmb{\tilde a}(\pmb{x})-\pmb{a}(\pmb{x}))\nabla v(\pmb{x})\,\mathrm{d}\pmb{x}
		\\&\quad+\int_{\Omega^\text{m}}(f(\pmb{x})-\tilde f(\pmb{x}))v(\pmb{x})\,\mathrm{d}\pmb{x}
	.\end{split}\]
	Denote $w=u_2-\tilde u$, and take $v=w$, inequalities analogous to~\eqref{eq:n6kf0t} gives
	\[\begin{split}
		&\quad\int_{\Omega^\text{m}}(\nabla w(\pmb{x}))^\mathrm{T}\pmb{a}(\pmb{x})\nabla w(\pmb{x})\,\mathrm{d}\pmb{x}
		\\&\le\|\pmb{\tilde a}-\pmb{a}\|_{L^\infty(\Omega^\text{m};\ell^2(\R^{d\times d}))}\|\tilde u\|_{H_0^1(\Omega^\text{m})}\|w\|_{H_0^1(\Omega^\text{m})}
		\\&\quad+\|f-\tilde f\|_{H^{-1}(\Omega^\text{m})}\|w\|_{H_0^1(\Omega^\text{m})}
	.\end{split}\]
	Assumption~\ref{ass:vardomEllip}(2) guarantees
	\[\int_{\Omega^\text{m}}(\nabla w(\pmb{x}))^\mathrm{T}\pmb{a}(\pmb{x})\nabla w(\pmb{x})\,\mathrm{d}\pmb{x}
		\ge\int_{\Omega^\text{m}}r(\nabla w(\pmb{x}))^\mathrm{T}\nabla w(\pmb{x})\,\mathrm{d}\pmb{x}
	=r\|w\|_{H_0^1(\Omega^\text{m})}^2 ,\]
	and thus
	\[\begin{split}
		\|w\|_{H_0^1(\Omega^\text{m})}&\le\frac1r\left(
		\|\pmb{\tilde a}-\pmb{a}\|_{L^\infty(\Omega^\text{m};\ell^2(\R^{d\times d}))}\|\tilde u\|_{H_0^1(\Omega^\text{m})}
		+\|f-\tilde f\|_{H^{-1}(\Omega^\text{m})}\right)
		\\&\le\frac1r\left(R(3C_J+1)\delta_\beta\cdot\frac{2C_\text{P}R}{r}+2R\delta_\beta\right)
		\\&=\frac{2R}{r^2}((3C_J+1)RC_\text{P}+r)\delta_\beta
	,\end{split}\]
	where we have used Proposition~\ref{prop:da_tildea}, \ref{prop:est_tildeuH01} and~\ref{prop:df_tildef}. 
	Combining with Corollary~\ref{cor:du_tildeu} gives
	\[\begin{split}
		\|u-u_2\|_{L^2(\Omega^\text{m})}&\le\|u-\tilde u\|_{L^2(\Omega^\text{m})}+\|\tilde u-u_2\|_{L^2(\Omega^\text{m})}
		\\&\le\|u-\tilde u\|_{L^2(\Omega^\text{m})}+C_\text{P}\|u_2-\tilde u\|_{H_0^1(\Omega^\text{m})}
		\\&\le\frac{2RC_\text{P}}{r^2}(3r+(3C_J+1)RC_\text{P})\delta_\beta
	.\qedhere\end{split}\]
\end{proof}

\subsubsection{Proofs for the Fixed Parameter Case}
\begin{proof}[Proof of Proposition~\ref{prop:eVDtransLip}]
	Given $\lambda,\bar\lambda\in\Lambda$, $\lambda\ne\bar\lambda$, the aim is to show $\|u^{\lambda}-u^{\bar\lambda}\|_{L^2(\Omega^\text{m})}\le C\|\lambda-\bar\lambda\|_2$. 
	We devide the interval $[0,1]$ into $K$ segments with endpoints $0=t_0<t_1<\cdots<t_K=1$, satisfying $2\|\lambda-\bar\lambda\|_2(t_k-t_{k-1})\le\delta_0$ for $k=1,\dots,K$. 
	Denote for simplicity $\Omega_k=\Omega_{\lambda+t_k(\bar\lambda-\lambda)}$, $T_k=T_{\lambda+t_k(\bar\lambda-\lambda)}$, $u_k=u^{\lambda+t_k(\bar\lambda-\lambda)}$ for $k=1,\dots,K$. 
	Given $k$, we take $\beta=T_k\circ T_{k-1}^{-1}$. Then $\beta$ is a $C^1$-diffeomorphism from $\bar\Omega^\text{m}$ onto itself, $\Omega_k=\beta(\Omega_{k-1})$. 
	For any $\pmb{x}\in\bar\Omega^\text{m}$, denote $\pmb{y}=T_{k-1}^{-1}(\pmb{x})$, we have
	\[\|\beta(\pmb{x})-\pmb{x}\|_2=\|T_k(\pmb{y})-T_{k-1}(\pmb{y})\|_2\le(t_k-t_{k-1})\|\lambda-\bar\lambda\|_2 \]
	by Assumption~\ref{ass:vardomTransforms}(3), and
	\[\begin{split}
		\left\lVert\frac{\partial \beta(\pmb{x})}{\partial \pmb{x}}-I_d\right\rVert_2
		&=\left\lVert \frac{\partial T_k(\pmb{y})}{\partial \pmb{y}}\left(\frac{\partial T_{k-1}(\pmb{y})}{\partial \pmb{y}}\right)^{-1}-I_d\right\rVert_2
		\\&\le\left\lVert \frac{\partial T_k(\pmb{y})}{\partial \pmb{y}}-\frac{\partial T_{k-1}(\pmb{y})}{\partial \pmb{y}}\right\rVert_2\cdot \left\lVert \left(\frac{\partial T_{k-1}(\pmb{y})}{\partial \pmb{y}}\right)^{-1}\right\rVert_2
		\\&=\sigma_{\min}\left(\frac{\partial T_{k-1}(\pmb{y})}{\partial \pmb{y}}\right)^{-1} \left\lVert \frac{\partial T_k(\pmb{y})}{\partial \pmb{y}}-\frac{\partial T_{k-1}(\pmb{y})}{\partial \pmb{y}}\right\rVert_2
		\\&\le 2\left\lVert \frac{\partial T_k(\pmb{y})}{\partial \pmb{y}}-\frac{\partial T_{k-1}(\pmb{y})}{\partial \pmb{y}}\right\rVert_2
		\\&\le 2(t_k-t_{k-1})\|\lambda-\bar\lambda\|_2
	,\end{split}\]
	where we have used (2), (3) of Assumption~\ref{ass:vardomTransforms}. 
	This assures $\delta_\beta\le 2(t_k-t_{k-1})\|\lambda-\bar\lambda\|_2\le\delta_0$, and Theorem~\ref{thm:vardomEllip} asserts
	\[\|u_{k-1}-u_k\|_{L^2(\Omega^\text{m})}\le C\delta_\beta/2\le C(t_k-t_{k-1})\|\lambda-\bar\lambda\|_2.\] 
	The conclusion follows by
	\[\begin{split}
		\|u^{\lambda}-u^{\bar\lambda}\|_{L^2(\Omega^\text{m})}&\le\sum_{k=1}^{K}\|u_{k-1}-u_k\|_{L^2(\Omega^\text{m})}
		\\&\le\sum_{k=1}^{K}C(t_k-t_{k-1})\|\lambda-\bar\lambda\|_2
		\\&=C\|\lambda-\bar\lambda\|_2
	.\qedhere\end{split}\]
\end{proof}
\begin{proof}[Proof of the assertion in Example~\ref{eg:eVDmovingDisk}]
	For this specific setting, we may take the parametrized domain transformation to be the rotation
	\[T_\lambda:\pmb{x}\mapsto R_\lambda \pmb{x}=
		\begin{bmatrix} \cos\lambda&-\sin\lambda\\\sin\lambda&\cos\lambda \end{bmatrix}
		\begin{bmatrix}x_1\\x_2\end{bmatrix}
	, \]
	which satisfies $\frac{\partial T_\lambda(\pmb{x})}{\partial \pmb{x}}=R_\lambda$ for any $\pmb{x}\in\bar\Omega^\text{m}$. 
	Then Assumption~\ref{ass:vardomTransforms} is satisfied, since we have $\sigma_{\min}(R_\lambda)=1$, $\|T_{\lambda}(\pmb{x})-T_{\bar\lambda}(\pmb{x})\|_2=
	\|R_{\lambda}\pmb{x}-R_{\bar\lambda}\pmb{x}\|_2=2\sin(\frac{|\lambda-\bar\lambda|}2)\|\pmb{x}\|_2\le |\lambda-\bar\lambda|$, 
	and
	\[\left\lVert \frac{\partial T_{\lambda}(\pmb{x})}{\partial \pmb{x}}-\frac{\partial T_{\bar\lambda}(\pmb{x})}{\partial \pmb{x}}\right\rVert_2
		=\|R_{\lambda}-R_{\bar\lambda}\|_2
		=\max_{\|\pmb{x}\|_2\le 1}\|R_{\lambda}\pmb{x}-R_{\bar\lambda}\pmb{x}\|_2
		\le|\lambda-\bar\lambda|
	.\]
	Theorem~\ref{thm:ellipVarDomDecWidthLip} then implies
	\[d_{1,C\pi}^\text{Deco}(\{u^\lambda|\lambda\in\Lambda\})\le Cd_{1,\pi}^\text{rDeco}(\Lambda) .\]
	By taking the decoder mapping to be $D=D_1\circ P_{[-1,1]}:Z=\R^1\to\Lambda$, where $D_1:Z_B=[-1,1]\to\Lambda$ is given by $z\mapsto\pi z$, we see that $D(Z_B)=\Lambda$, and $d_{1,\pi}^\text{rDeco}(\Lambda)=0$. 
	The assertion is proved. 
\end{proof}
\begin{proof}[Proof of the assertion in Example~\ref{eg:eVDmovingHole}]
	The proof is the same as Example~\ref{eg:eVDmovingDisk}. 
\end{proof}
\begin{proof}[Proof of the assertion in Example~\ref{eg:eVDdefmHole}]
	The domain transformations $T_\lambda:\bar\Omega^\text{m}\to\bar\Omega^\text{m}$ can be constructed as
	\[T_\lambda(\pmb{x})=(x_1+\lambda_1\sin x_1,x_2+\lambda_2\sin x_2),\quad \pmb{x}=(x_1,x_2)\in\bar\Omega^\text{m} ,\]
	which satisfies
	\[
		\frac{\partial T_\lambda(\pmb{x})}{\partial \pmb{x}}=\mathrm{diag}(1+\lambda_1\cos x_1,1+\lambda_2\cos x_2)
	,\]
	and it is easy to check that Assumption~\ref{ass:vardomTransforms} holds. 
	Theorem~\ref{thm:ellipVarDomDecWidthLip} implies
	\[d_{2,C/\sqrt 2}^\text{Deco}(\{u^\lambda|\lambda\in\Lambda\})\le Cd_{2,1/\sqrt 2}^\text{rDeco}(\Lambda) .\]
	We take $Z=\R^2$, and $Z_\Lambda=[-1/\sqrt 2,1/\sqrt 2]^2$ which is a convex subset of $Z_B$. 
	The mapping $D_1:Z\to\R^2,z\mapsto z/\sqrt 2$ satisfies $D_1(Z_\Lambda)=\Lambda$. 
	By taking the decoder mapping to be $D=D_1\circ P_{Z_\Lambda}$, we conclude that $d_{2,1/\sqrt 2}^\text{rDeco}(\Lambda)=0$, and hence $d_{2,C/\sqrt 2}^\text{Deco}(\{u^\lambda|\lambda\in\Lambda\})=0$. 
\end{proof}
\begin{proof}[Proof of the assertion in Example~\ref{eg:eVDctvc}]
	In this setting, the domain transformation can be taken as
	$T_\lambda(\pmb{x})=(x_1,x_2+\rho_\lambda(x_1)\sin x_2)$ for $\pmb{x}\in\bar\Omega^\text{m}$, with Jacobian matrix
	\[\frac{\partial T_\lambda(\pmb{x})}{\partial \pmb{x}}=
		\begin{bmatrix} 1\\\rho_\lambda'(x_1)\sin x_2&1+\rho_\lambda(x_1)\cos x_2 \end{bmatrix} .\]
	Validation of Assumption~\ref{ass:vardomTransforms}(1) is straightforward. 
	As for (2) and (3), we note that
	\[\begin{split}
		\rho_\lambda'(x_1)^2&=\Bigl(\sum_{k=1}^{\infty}\frac{\sqrt 6}{k\pi}(-a_k\sin kx_1+b_k\cos kx_1)\Bigr)^2
		\\&\le\Bigl(\sum_{k=1}^{\infty}\frac{6}{k^2\pi^2}\Bigr)\Bigl(\sum_{k=1}^{\infty}(-a_k\sin kx_1+b_k\cos kx_1)^2\Bigr)
		\\&\qquad\text{(by Cauchy's inequality)}
		\\&=\sum_{k=1}^{\infty}|-a_k\sin kx_1+b_k\cos kx_1|^2
		\\&\le\sum_{k=1}^{\infty}\left(\sqrt{a_k^2+b_k^2}\right)^2
		=\|\lambda\|^2
	,\end{split}\]
	and similarly $|\rho_\lambda(x_1)|\le\|\lambda\|$ holds as well. 
	Then
	\[T_\lambda(\pmb{x})-T_{\bar\lambda}(\pmb{x})=(0,(\rho_\lambda(x_1)-\rho_{\bar\lambda}(x_1))\sin x_2)
		=(0,\rho_{\lambda-\bar\lambda}(x_1)\sin x_2) ,\]
	\[\|T_\lambda(\pmb{x})-T_{\bar\lambda}(\pmb{x})\|_2\le|\rho_{\lambda-\bar\lambda}(x_1)|\cdot|\sin x_2|
		\le\|\lambda-\bar\lambda\| ,\]
	and
	\[ \frac{\partial T_{\lambda}(\pmb{x})}{\partial \pmb{x}}-\frac{\partial T_{\bar\lambda}(\pmb{x})}{\partial \pmb{x}}
		=\begin{bmatrix} 0&0\\\rho_{\lambda-\bar\lambda}'(x_1)\sin x_2&\rho_{\lambda-\bar\lambda}(x_1)\cos x_2 \end{bmatrix} ,\]
	\[\begin{split}
		\left\lVert \frac{\partial T_{\lambda}(\pmb{x})}{\partial \pmb{x}}-\frac{\partial T_{\bar\lambda}(\pmb{x})}{\partial \pmb{x}}\right\rVert_2^2
		&=\bigl(\rho_{\lambda-\bar\lambda}'(x_1)\sin x_2\bigr)^2+\bigl(\rho_{\lambda-\bar\lambda}(x_1)\cos x_2\bigr)^2
		\\&\le\|\lambda-\bar\lambda\|^2(\sin^2x_2+\cos^2x_2)
		\\&=\|\lambda-\bar\lambda\|^2
	,\end{split}\]
	indicating that Assumption~\ref{ass:vardomTransforms}(3) holds. 
	Furthermore, note that $w_k\le w_1=1$ implies
	\[
		\|\lambda\|^2=\sum_{k=1}^{\infty}(a_k^2+b_k^2)
		\le\sum_{k=1}^{\infty}\frac{a_k^2+b_k^2}{w_k}
		\le(3/8)^2
	.\]
	Given $p,q\in\R$, 
	\[\begin{split}
		\sigma_{\min}([\begin{smallmatrix} 1\\q&p \end{smallmatrix}])^2
		&=\lambda_{\min}([\begin{smallmatrix} 1\\q&p \end{smallmatrix}][\begin{smallmatrix} 1&q\\&p \end{smallmatrix}])
		=\lambda_{\min}([\begin{smallmatrix} 1&q\\q&p^2+q^2 \end{smallmatrix}])
		\\&=\frac12\bigl(p^2+q^2+1-\sqrt{(p^2+q^2+1)^2-4p^2}\bigr)
	,\end{split}\]
	and $\sigma_{\min}([\begin{smallmatrix} 1\\q&p \end{smallmatrix}])\ge 1/2$
	if and only if $p^2\ge q^2/3+1/4$. 
	For any $x\in\bar\Omega^\text{m}$, we let $p=1+\rho_\lambda(x_1)\cos x_2$, $q=\rho_\lambda'(x_1)\sin x_2$. 
	Then
	\[p\ge 1-|\rho_\lambda(x_1)|\ge 1-\|\lambda\|\ge 5/8,\]
	\[|q|\le|\rho_\lambda'(x_1)|\le\|\lambda\|\le 3/8,\]
	and $p^2\ge q^2/3+1/4$ indeed holds. 
	This validates Assumption~\ref{ass:vardomTransforms}(2) as $\sigma_{\min}(\frac{\partial T_\lambda(x)}{\partial x})=\sigma_{\min}([\begin{smallmatrix} 1\\q&p \end{smallmatrix}])\ge 1/2$. 

	Theorem~\ref{thm:ellipVarDomDecWidthLip} then implies
	\[d_{2n,3C/8}^\text{Deco}(\{u^\lambda|\lambda\in\Lambda\})\le Cd_{2n,3/8}^\text{rDeco}(\Lambda) .\]
	To estimate the latter term, we take $Z=\R^{2n}$, and let $D_1:z\mapsto\lambda=(a_1,b_1,a_2,\dots)$ be given as $a_k=3z_{2k-1}/8$, $b_k=3z_{2k}/8$ for $k=1,\dots,n$, $a_k=b_k=0$ for $k>n$. 
	Then the mapping $D_1$ is $3/8$-Lipschitz. 
	Each $\lambda=(a_1,b_1,a_2,\dots)\in\Lambda$ can be associated with a vector $z$ such that $z_{2k-1}=8a_k/3$, $z_{2k}=8b_k/3$, $k=1,\dots,n$. 
	Then $z\in Z_B$ since $\|z\|\le 8\|\lambda\|/3\le 1$, and
	\[\begin{split}
		\lambda-D_1(z)&=(0,\dots,0,a_{n+1},b_{n+1},a_{n+2},\dots),
		\\\|\lambda-D_1(z)\|^2&=\sum_{k>n}(a_k^2+b_k^2)
		\le w_n\sum_{k>n}\frac{a_k^2+b_k^2}{w_k}
		\\&\le w_n\sum_{k=1}^\infty\frac{a_k^2+b_k^2}{w_k}
		\le \frac{9w_n}{64}
	.\end{split}\]
	Now $Z_\Lambda=D_1^{-1}(\Lambda)$ is a convex subset of $Z_B$. 
	Taking $D=D_1\circ P_{Z_\Lambda}$, we have
	\[\sup_{\lambda\in\Lambda}\inf_{z\in Z_B}\|\lambda-D(z)\|\le 3\sqrt{w_n}/8,\]
	and $D(Z)\subset\Lambda$ holds. 
	This indicates $d_{2n,3/8}^\text{rDeco}(\Lambda)\le 3\sqrt{w_n}/8$ by the definition of the restricted decoder width, and we conclude that $d_{2n,3C/8}^\text{Deco}(\{u^\lambda|\lambda\in\Lambda\})\le 3C\sqrt{w_n}/8$. 
\end{proof}

\subsubsection{Proofs for the Variable Parameter Case}
\begin{proof}[Proof of Proposition~\ref{prop:eVDPLip}]
	By Proposition~\ref{prop:ellpLip}, we have 
	\[\begin{split}
		\|u^{(\pmb{a},f,\bar\lambda)}-u^{(\pmb{\bar a},\bar f,\bar\lambda)}\|_{L^2(\Omega^\text{m})}
		&\le C_\text{P}\|u^{(\pmb{a},f,\bar\lambda)}-u^{(\pmb{\bar a},\bar f,\bar\lambda)}\|_{H_0^1(\Omega^\text{m})}
		\\&\le \frac{C_\text{P}^2}{r^2}\Bigl(
			R\|\pmb{\bar a}-\pmb{a}\|_{L^\infty(\Omega^\text{m};\ell^2(\R^{d\times d}))}
			+r\|\bar f-f\|_{L^2} \Bigr)
		\\&\le C_1\Bigl(
			\|\pmb{\bar a}-\pmb{a}\|_{L^\infty(\Omega^\text{m};\ell^2(\R^{d\times d}))}
			+\|\bar f-f\|_{L^2} \Bigr)
	.\end{split}\]
	By Proposition~\ref{prop:eVDtransLip}, we have 
	\[\begin{split}
		\|u^{(\pmb{a},f,\lambda)}-u^{(\pmb{a},f,\bar\lambda)}\|_{L^2(\Omega^\text{m})}
		\le C\|\lambda-\bar\lambda\|_2
		\le C_1\|\lambda-\bar\lambda\|_2
	.\end{split}\]
	A direct combination of these two inequalities shall give
	\[\begin{split}
		\|u^{(\pmb{a},f,\lambda)}-u^{(\pmb{\bar a},\bar f,\bar\lambda)}\|_{L^2(\Omega^\text{m})}
		&\le\|u^{(\pmb{a},f,\lambda)}-u^{(\pmb{a},f,\bar\lambda)}\|_{L^2(\Omega^\text{m})}
			+\|u^{(\pmb{a},f,\bar\lambda)}-u^{(\pmb{\bar a},\bar f,\bar\lambda)}\|_{L^2(\Omega^\text{m})}
		\\&\le C_1d_{\mathcal{A}_0}((\pmb{a},f,\lambda),(\pmb{\bar a},\bar f,\bar\lambda))
	.\qedhere\end{split}\]
\end{proof}
\begin{proof}[Proof of the assertion in Example~\ref{eg:eVPD}]
	The proof is similar to that of Example~\ref{eg:parabEg1}. 
	The PDE coefficients are now parametrized by the mapping $S_\Lambda:\Lambda_\mu\times\Lambda\to\mathcal{A}_0$, $(\mu,\lambda)\mapsto(\pmb{a},f,\lambda)$. 
	In assumption (c), the condition $\sum_{k=1}^{K}(\mu_k^a)^2\le\frac14$ implies $|\mu_k^a|\le\frac12$ for $k=1,\dots,K$. 
	By assumption (a), it is easy to see that all possible $\pmb{a}$'s are 2-Lipschitz continuous, and $\pmb{a}\in\Sigma_{1,2}^d(\Omega^\text{m})$, thereby $\pmb{a}\in \mathsf{D}_{1,2}^a$. 
	For a different set of coefficients $(\pmb{\bar a},\bar f,\bar\lambda)=S_\Lambda(\bar\mu,\bar\lambda)$, an argument similar to~\eqref{eq:parabDiff_b} shows
	\[\begin{split}
		\|\pmb{a}-\pmb{\bar a}\|_{L^\infty(\Omega^\text{m};\ell^2(\R^d\times\R^d))}
		&=\max_{1\le k\le K}\left\lVert(\mu_k^a-\bar\mu_k^a)\phi_k^a\right\rVert_{L^\infty(\Omega^\text{m})}
		\\&\le\max_{1\le k\le K}\|\mu-\bar\mu\|_2\cdot\left\lVert\phi_k^a\right\rVert_{L^\infty(\Omega^\text{m})}
		\\&\le\|\mu-\bar\mu\|_2
	.\end{split}\]
	As for $f$, we have $f\in \mathsf{D}_2^f$ since
	$\|f\|_{L^2}^2\le\left(\sum_{k=1}^{K}(\mu_k^f)^2\right)\left(\sum_{k=1}^{K}\|\phi_k^f\|_{L^2}^2\right)\le\frac14\cdot 1\le 4$ 
	(similar to~\eqref{eq:parabEg1FnormEst}), and
	\[ \|f-\bar f\|_{L^2}^2\le\left(\sum_{k=1}^{K}(\mu_k^f-\bar\mu_k^f)^2\right)\left(\sum_{k=1}^{K}\|\phi_k^f\|_{L^2}^2\right)\le\|\mu-\bar\mu\|_2^2\cdot 1 .\]
	Viewing $(\mu,\lambda)$ as a vector in $\R^{2K+2}$, we apply the definition of the metric on $\mathcal{A}_0$ (as given in~\eqref{eq:eVDPdefA0metric}) to get
	\[d_{\mathcal{A}_0}((\pmb{a},f,\lambda),(\pmb{\bar a},\bar f,\bar\lambda)) \le
		\|\bar\mu-\mu\|_2
		+\|\bar\mu-\mu\|_2
		+\|\bar\lambda-\lambda\|_2
		\le 3\|(\bar\mu,\bar\lambda)-(\mu,\lambda)\|_2
		,\]
	and the mapping $S_\Lambda:(\mu,\lambda)\mapsto(\pmb{a},f,\lambda)$ is therefore 3-Lipschitz continuous. 
	Furthermore, assumption (c) as well as the definition of $\Lambda$ implies
	\[\|(\mu,\lambda)\|_2^2=\sum_{k=1}^{K}(\mu_k^a)^2+\sum_{k=1}^{K}(\mu_k^f)^2+(\lambda_1^2+\lambda_2^2)\le\frac14+\frac14+\left(\frac14+\frac14\right)=1 \]
	for any $(\mu,\lambda)\in\Lambda_\mu\times\Lambda$, and hence $d_{2K+2,1}^\text{rDeco}(\Lambda_\mu\times\Lambda)=0$ since we can take $Z=\R^{2K+2}$, $D=P_{\Lambda_\mu\times\Lambda}$. 
	Proposition~\ref{prop:decWidthLip} then implies $d_{2K+2,3}^\text{Deco}(S_\Lambda(\Lambda_\mu\times\Lambda))=0$, and we may use Theorem~\ref{thm:eVDPwidthFromA} to conclude the proof
	(note that $\{u^{\mu,\lambda}|\mu\in\Lambda_\mu,\lambda\in\Lambda\}=S_0(S_\Lambda(\Lambda_\mu\times\Lambda))$ in this example). 
\end{proof}

\subsection{Advection Equations on a Fixed Domain}
\begin{proof}[Proof of Proposition~\ref{prop:advEg1Holder}]
	We have
	\[\begin{split}
		\left\|u-\bar u\right\|_{L^p}^p &=\int_{-1}^1\mathrm{d}x\int_0^1\left|u(t,x)-\bar u(t,x)\right|^p\,\mathrm{d}t
		\\&=\int_{-1}^1\mathrm{d}x\int_0^1\left|\chi_{\{t\le\int_0^x\frac{\mathrm{d}s}{c(s)}\}}-\chi_{\{t\le\int_0^x\frac{\mathrm{d}s}{\bar c(s)}\}}\right|^p\,\mathrm{d}t
		\\&\le\int_{-1}^1\mathrm{d}x\int_{-\infty}^{+\infty}\left|\chi_{\{t\le\int_0^x\frac{\mathrm{d}s}{c(s)}\}}-\chi_{\{t\le\int_0^x\frac{\mathrm{d}s}{\bar c(s)}\}}\right|^p\,\mathrm{d}t
		\\&=\int_0^1\mathrm{d}x\int_{-\infty}^{+\infty}\left|\chi_{\{t\le\int_0^x\frac{\mathrm{d}s}{c(s)}\}}-\chi_{\{t\le\int_0^x\frac{\mathrm{d}s}{\bar c(s)}\}}\right|\,\mathrm{d}t
		\\&=\int_0^1\left|\int_0^x\frac{\mathrm{d}s}{c(s)}-\int_0^x\frac{\mathrm{d}s}{\bar c(s)}\right|\,\mathrm{d}x
		\\&\le\int_0^1\mathrm{d}x\int_0^x\left|\frac{1}{c(s)}-\frac{1}{\bar c(s)}\right|\mathrm{d}s
		\\&\le\int_0^1\mathrm{d}x\int_{-1}^1\left|\frac{1}{c(s)}-\frac{1}{\bar c(s)}\right|\mathrm{d}s
		\\&\le\int_{-1}^1\frac{|\bar c(s)-c(s)|}{r^2}\mathrm{d}s
		\\&=r^{-2}\|c-\bar c\|_{L^1([-1,1])}
	.\qedhere\end{split}\]
\end{proof}
\subsubsection{The Case of Using $L^1$-norm}
\begin{lem}\label{lem:decW_GRF}
	Assume $\mathcal{V}$ is a Banach space. 
	Given $\phi_0,\dots,\phi_K\in\mathcal{V}$, 
	we let
	\[\mathcal{K}_\mathcal{V}=\left\{\phi_0+\sum_{k=1}^{K}\lambda_k\phi_k\ \middle|\ \lambda\in[-1,1]^K\right\}\subset\mathcal{V}.\]
	Then
	$d_{K,l}^\text{rDeco}(\mathcal{K}_\mathcal{V})=0$
	holds, where $l=\sqrt{K\sum_{k=1}^{K}\mu_k^2}$\ , $\mu_k=\|\phi_k\|_\mathcal{V}>0$. 
\end{lem}
\begin{proof}
	The proof has much in common with that of Example~\ref{eg:ellipEg1}. 
	For any $v=\phi_0+\sum_{k=1}^{K}\lambda_k\phi_k\in \mathcal{K}_\mathcal{V}$, $\bar v=\phi_0+\sum_{k=1}^{K}\bar\lambda_k\phi_k\in \mathcal{K}_\mathcal{V}$, we have
	\begin{equation}\label{eq:lemDecW_GRF_D1Lip}\begin{split}
		\|v-\bar v\|_{\mathcal{V}}
		&=\left\lVert \sum_{k=1}^{K}(\lambda_k-\bar\lambda_k)\phi_k\right\rVert_{\mathcal{V}}
		\le\sum_{k=1}^{K}|\lambda_k-\bar\lambda_k|\cdot\|\phi_k\|_{\mathcal{V}}
		\\&=\sum_{k=1}^{K}|\lambda_k-\bar\lambda_k|\mu_k
		\le\sqrt{\sum_{k=1}^{K}|\lambda_k-\bar\lambda_k|^2}\sqrt{\sum_{k=1}^{K}\mu_k^2}
		\\&=\|\lambda-\bar\lambda\|_2\sqrt{\sum_{k=1}^{K}\mu_k^2}
	.\end{split}\end{equation}
	We consider the mapping
	\[D_1:Z=\R^K\to\mathcal{V},\quad z\mapsto v=\phi_0+\sqrt K\sum_{k=1}^{K}z_k\phi_k .\]
	By~\eqref{eq:lemDecW_GRF_D1Lip}, $D_1$ is $\sqrt{K\sum_{k=1}^{K}\mu_k^2}$-Lipschitz continuous. 
	Now $Z_1=[-1/\sqrt K,1/\sqrt K]^K$ is a convex subset of $Z_B$, and $D_1(Z_1)=\mathcal{K}_\mathcal{V}$ holds. 
	We take $D=D_1\circ P_{Z_1}$, whose Lipschitz constant is also $l=\sqrt{K\sum_{k=1}^{K}\mu_k^2}$. 
	The lemma then follows directly by the definition of the restricted decoder width. 
\end{proof}
\begin{proof}[Proof of the assertion in Example~\ref{eg:advEg1_L1}]
	Denote $\mathcal{A}=\{c^\lambda\mid\lambda\in\Lambda\}$. 
	According to Lemma~\ref{lem:decW_GRF} and the definition of $c^\lambda$, we see that 
	\[d_{K,l}^\text{rDeco}(\mathcal{A})_{L^1([-1,1])}=0\]
	holds with
	$l=\sqrt{K\sum_{k=1}^{K}\|\phi_k\|_{L^1([-1,1])}^2}$. 
	We may apply Theorem~\ref{thm:advL1DecKvsA} to conclude that
	$d_{K,l/r^2}^\text{Deco}\left(\left\{u^\lambda\ \middle|\ \lambda\in\Lambda\right\}\right)_{L^1}=0.$
\end{proof}
\subsubsection{The Case of Using $L^2$-norm}
Similar to Proposition~\ref{prop:decWidthLip}, we have the following lemma for the entropy number:
\begin{lem}\label{lem:ent_Holder}
	Let $\mathcal{V,U}$ be two Banach spaces, $\mathcal{K_V}\subset\mathcal{V}$ be a compact subset, and $\mathcal{F}:\mathcal{K_V}\to\mathcal{U}$ be a continuous mapping satisfying
	\[\|\mathcal{F}(v)-\mathcal{F}(\bar v)\|_{\mathcal{U}}\le w(\|v-\bar v\|_{\mathcal{V}})\]
	for all $v,\bar v\in\mathcal{K_V}$, where $w:[0,+\infty)\to[0,+\infty)$ is a non-decreasing continuous function. 
	Then, denoting $\mathcal{K=F(K_V)}$, we shall have
	$\tilde\epsilon_n(\mathcal{K})\le w(\tilde\epsilon_n(\mathcal{K_V})) .$
	The obvious corollary is
	$\epsilon_n(\mathcal{K})\le w(\tilde\epsilon_n(\mathcal{K_V})) .$
\end{lem}
\begin{proof}
	For $V_n\subset\mathcal{K_V}$, we have
	\[\begin{split}
		\sup_{u\in \mathcal{F}(\mathcal{K_V})}\inf_{\bar u\in \mathcal{F}(V_n)}\|u-\bar u\|_{\mathcal{U}}
		&=\sup_{v\in \mathcal{K_V}}\inf_{\bar v\in V_n}\|\mathcal{F}(v)-\mathcal{F}(\bar v)\|_{\mathcal{U}}
		\\&\le\sup_{v\in \mathcal{K_V}}\inf_{\bar v\in V_n}w(\|v-\bar v\|_{\mathcal{V}})
		\\&=w\Bigl(\sup_{v\in \mathcal{K_V}}\inf_{\bar v\in V_n}\|v-\bar v\|_{\mathcal{V}}\Bigr)
	,\end{split}\]
	where the last equality is due to the monotonicity and the continuity of $w(\cdot)$. 
	Since $V_n\subset\mathcal{K_V},\mathcal{F}(V_n)\subset\mathcal{K}$ and $\#\mathcal{F}(V_n)\le\#V_n\le 2^n$, the lemma follows directly from the definition of the inner entropy number. 
\end{proof}
\begin{proof}[Proof of Theorem~\ref{thm:advL2EntKvsA}]
	Taking $p=2$ in Proposition~\ref{prop:advEg1Holder}, we see that 
	the solution mapping $S_0|_{\mathcal{A}}:\mathcal{A}\to\mathcal{U}=L^2([0,1]\times[-1,1])$ satisfies $\|S_0(c)-S_0(\bar c)\|_\mathcal{U}\le w(\|c-\bar c\|_{L^1([-1,1])})$, where $w(\epsilon)=r^{-1}\sqrt\epsilon$. 
	Lemma~\ref{lem:ent_Holder} then implies $\epsilon_n(S_0(\mathcal{A}))\le r^{-1}\sqrt{\tilde\epsilon_n(\mathcal{A})}$

	If $c,\bar c\in\mathcal{A}$, $u=S_0(c),\bar u=S_0(\bar c)$, we have
	\[\begin{split}
		\left\|u-\bar u\right\|_\mathcal{U}^2
		&=\int_{-1}^1\mathrm{d}x\int_0^1\left|\chi_{\{t\le\int_0^x\frac{\mathrm{d}s}{c(s)}\}}-\chi_{\{t\le\int_0^x\frac{\mathrm{d}s}{\bar c(s)}\}}\right|^2\,\mathrm{d}t
		\\&\le\int_{-1}^1\mathrm{d}x\int_0^11\,\mathrm{d}t
		\\&=2
	,\end{split}\]
	indicating $\mathrm{diam}(S_0(\mathcal{A}))\le\sqrt 2<2$. 
	Corollary~\ref{cor:decoVSentNum} then gives
	$
		d_{26n,4\cdot 2}^\mathrm{Deco}(S_0(\mathcal{A}))
		\le 3\epsilon_n(S_0(\mathcal{A}))
	,$
	and we conclude that 
	$d_{26n,8}^\mathrm{Deco}(S_0(\mathcal{A}))\le 3r^{-1}\sqrt{\tilde\epsilon_n(\mathcal{A})}$
\end{proof}
\begin{lem}\label{lem:ent_GRF}
	Let $\mathcal{V}$ be a Banach space. Given $\phi_0,\dots,\phi_K\in\mathcal{V}$, we let 
	\[\mathcal{K_V}=\left\{\phi_0+\sum_{k=1}^{K}\lambda_k\phi_k\ \middle|\ \lambda\in[-1,1]^K\right\}\subset\mathcal{V}\]
	as in Lemma~\ref{lem:decW_GRF}.
	Then
	$\tilde\epsilon_n(\mathcal{K_V})\le C_12^{-n/K}$
	holds for sufficiently large $n$, where $C_1=4K\left(\prod_{k=1}^K\mu_k\right)^{1/K}$, $\mu_k=\|\phi_k\|_\mathcal{V}>0$. 
\end{lem}
\begin{proof}
	Given $n$, we take $\delta_n=2^{-n/K}\left(\prod_{k=1}^K\mu_k\right)^{1/K}$, and assume $n$ is large enough such that $\delta_n<\frac12\mu_k$ holds for $k=1,\dots,K$. 
	Let $m_k=\left\lfloor \mu_k/\delta_n\right\rfloor$. 
	The interval $[-1,1]$ is split into $m_k$ disjoint sub-intervals with equal length, and the midpoints of these sub-intervals form the set
	\begin{equation*}
		A_k=\left\{\frac{2i}{m_k}-1-\frac1{m_k}\ \middle|\ i=1,\dots,m_k\right\} 
	.\end{equation*}
	The number of points contained in $A=\prod_{k=1}^KA_k\subset[-1,1]^K$ is
	\[\#A=\prod_{k=1}^Km_k\le\prod_{k=1}^K\frac{\mu_k}{\delta_n}=2^n .\]

	Taking $V_n=\left\{\phi_0+\sum_{k=1}^{K}\lambda_k\phi_k\mid\lambda\in A\right\}\subset\mathcal{K_V}$, 
	it can be shown that
	\begin{equation}\label{eq:n6c96j}
		\sup_{v\in \mathcal{K_V}}\inf_{\bar v\in V_n}\|v-\bar v\|_\mathcal{V}\le\sum_{k=1}^{K}\frac{\mu_k}{m_k} .
	\end{equation}
	In fact, for any $v=\phi_0+\sum_{k=1}^{K}\lambda_k\phi_k\in \mathcal{K_V}$, we may find $\bar\lambda_k\in A_k$ with $|\bar\lambda_k-\lambda_k|\le\frac1{m_k}$, $k=1,\dots,K$ by the definition of each $A_k$. 
	Then $\bar v=\phi_0+\sum_{k=1}^{K}\bar\lambda_k\phi_k$ lies in $V_n$, and
	\[\|v-\bar v\|_\mathcal{V}\le\sum_{k=1}^{K}|\lambda_k-\bar\lambda_k|\cdot\|\phi_k\|_\mathcal{V}\le\sum_{k=1}^{K}\frac{\mu_k}{m_k} .\]
	We get~\eqref{eq:n6c96j} since $v$ is arbitrary. 

	Now $\tilde\epsilon_n(\mathcal{K_V})\le\sum_{k=1}^{K}\frac{\mu_k}{m_k}$ by the definition of the inner entropy number. 
	Since $m_k=\left\lfloor \mu_k/\delta_n\right\rfloor\ge\frac{\mu_k}{\delta_n}-1$, we have
	\[\tilde\epsilon_n(\mathcal{K_V})\le\delta_n\sum_{k=1}^{K}\frac{\mu_k}{\mu_k-\delta_n} .\]
	Note that
	\[\lim_{n\to\infty}\delta_n=0,\quad\lim_{n\to\infty}\sum_{k=1}^{K}\frac{\mu_k}{\mu_k-\delta_n}=\sum_{k=1}^{K}\frac{\mu_k}{\mu_k}=K ,\]
	and thus $\tilde\epsilon_n(\mathcal{K_V})\le 4K\delta_n$ holds for large enough $n$. 
	This proves the lemma. 
\end{proof}
\begin{proof}[Proof of the assertion in Example~\ref{eg:advEg1}]
	Again, we denote $\mathcal{A}=\{c^\lambda\mid\lambda\in\Lambda\}$. 
	According to Lemma~\ref{lem:decW_GRF} and the definition of $c^\lambda$, we see that 
	$\tilde\epsilon_n(\mathcal{A})\le C_12^{-n/K}$ for sufficiently large $n$, where
	$C_1=4K\left(\prod_{k=1}^{K}\|\phi_k\|_{L^1([-1,1])}\right)^{1/K} .$
	The assertion then follows directly by applying Theorem~\ref{thm:advL2EntKvsA}. 
\end{proof}

\section{Discussions and Conclusions}\label{sec:conclusion}
\begin{figure}[tb]
	\centering
	\includegraphics[width=0.2\linewidth]{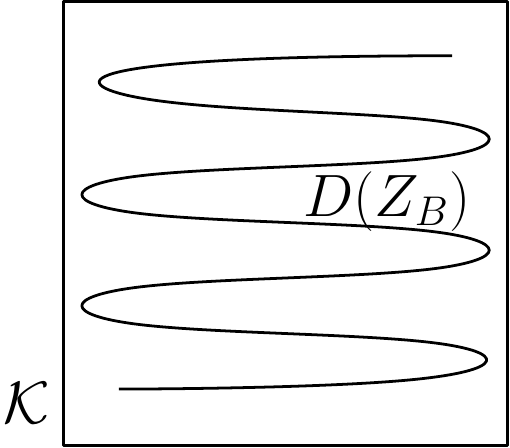}
	\caption{
	  If we do not impose the $l$-Lipschitz continuity constraint on the decoder mapping $D$, the resulting trial manifold $D(Z_B)$ may resemble a space-filling curve. 
	}%
	\label{fig:spFillCurve}
\end{figure}
In the definition of the decoder width, we impose the $l$-Lipschitz continuity condition as well as the restriction $\|\pmb{z}\|_2\le 1$. 
Both of them are required to give a reasonable notion of width. 
If either of the conditions is dropped, the width could become identically zero even for latent dimension $n=1$, since $D(Z_B)$ \verB{or $D(Z)$} may resemble the space-filling curve as illustrated in Figure~\ref{fig:spFillCurve}. 
In such cases, the corresponding ROM would indeed have a strong representational power, and be capable of approximating any desired solutions to the parametric PDEs. 
However, it leaves all the challenges to the online process of finding the latent vector $\pmb{z}$, as the corresponding optimization problem has become highly non-convex. 
This makes the ROM unuseful in the realistic settings. 

We also note that, although these two conditions are theoretically essential, they might not be introduced explicitly in the numerical implementation of the MAD method. 
As we have explained in~\cite{Ye2023MetaAD}, the hard constraint $\|\pmb{z}\|_2\le 1$ can be relaxed to a soft constraint, i.e. an additional penalty term in the objective function during the offline pre-training and the online fine-tuning stages. 
Furthermore, we represent the decoder mapping $D:Z\to\mathcal{U}$ as a neural network trained via stochastic gradient descent (SGD), 
which is typically known to exhibit the implicit regularization phenomenon~\cite{Neyshabur2014InSO,Neyshabur2017ExploringGD}. 
Even if the explicit Lipschitz continuity constraint is not imposed on the decoder mapping during the offline pre-training stage, 
we would expect to obtain a $D$ without the ill-posedness that is theoretically possible. 
Further numerical verification may be helpful in this respect, and is left for future work. 

The current work admits many possible extensions, 
especially when we focus on the parametric PDEs on a variable domain. 
While the analysis in this paper is limited to the elliptic case, it may be extended to a broader range of PDEs, 
including the parabolic equations and the wave equations. 
Furthermore, our Assumption~\ref{ass:vardomTransforms} requires 
a shared topological structure among all possible domains $\Omega$.
For scenarios where the domain topology changes, 
the MAD method (as a numerical solver) is still applicable due to its flexibility, 
but a theoretical estimation of the corresponding decoder width remains to be developed. 
As a reference domain is no longer available in this setting, new methods of analysis have to be introduced. 

In sum, 
this paper analyzes the basic properties of the decoder width, which serves as a quantified criterion for the best possible performance of the MAD method. 
Theoretical estimations of the decoder width for several parametric PDEs are provided, including the advection equations (Section~\ref{sec:advVC}) where classical linear ROMs appear to be inefficient, and the elliptic equations on a variable domain (Section~\ref{sec:ellpVD}) which could bring inconvenience for other deep autoencoder-based ROMs~\cite{Lee2020ModelRD,Fresca2021AComprehensiveDL,Fresca2022DeepLB,Gruber22AComparisonNN}. 
These results indicate the great potential of the MAD method for solving these parametric PDEs.

\bibliographystyle{amsplain}
\bibliography{dec-width}
\appendix
\section{Nomenclature}
\begin{table}[htpb]
	\centering
	\caption{Summary of the commonly used notations in this work. The specific meaning of the symbols may be different depending on the context. }
	\hspace*{-\textwidth}
	\begin{tabular}{ll}
	\hline
	Notation & Description\\
	\hline
		$\Omega\subset\R^d$ & the computational domain of the PDEs\\
		$u:\Omega\to\R^{d_u}$ & the solution to the PDEs (examples in this paper involve only $d_u=1$)\\
		$\mathcal{U}=\mathcal{U}(\Omega;\R^{d_u})$ & a Banach space in which the solutions lie\\
		$\Omega_T=[0,T]\times\Omega$ & the computational domain of the parabolic PDEs (Section~\ref{sec:parabEqn})\\
		$\Omega^\text{m}\supset\Omega$ & the master domain for the PDEs on a variable domain (Section~\ref{sec:ellpVD})\\
		$C_\text{P}$ & the Poincar\'e constant of $\Omega$ (fixed domain) or $\Omega^\text{m}$ (variable domain)\\
		$\eta$ & the variable parameter of the PDEs\\
		$\mathcal{A}$ & the space of the variable PDE parameters\\
		$S:\mathcal{A}\to\mathcal{U}$ & the solution mapping of the parametric PDEs\\
		$\mathcal{K}=S(\mathcal{A})$ & the solution set\\
		$\mathcal{A}_0\supset\mathcal{A}$ & the extended space of the PDE parameters\\
		$d_{\mathcal{A}_0}(\cdot,\cdot)$ & the metric on $\mathcal{A}_0$, typically inherited from the norm of a Banach space\\
		$S_0:\mathcal{A}_0\to\mathcal{U}$ & the extended solution mapping, satisfying $S=S_0|_{\mathcal{A}}$\\
		$S_\Lambda:\Lambda\to\mathcal{A}_0$ & parametrization of $\mathcal{A}$, satisfying $\mathcal{A}=S_\Lambda(\Lambda)$\\
		$Z=\R^n$ & the latent space, endowed with $\ell^2$-norm unless specified otherwise\\
		$Z_B$ & the closed unit ball of $Z$\\
		$D:Z\to\mathcal{U}$ & the decoder mapping\\
		$d_{n,l}^\text{Deco}(\cdot)$ & the decoder width~\eqref{eq:decWidth}\\
		$d_{n,l}^\text{rDeco}(\cdot)$ & the restricted decoder width defined in Section~\ref{sec:decWidthFD}\\
		$\epsilon_n(\cdot),\tilde\epsilon_n(\cdot)$ & the entropy number~\eqref{eq:ent_num_def}, and the inner entropy number\\
		$\bar B_R(\mathcal{V})$ & the closed ball of radius $R>0$ in the Banach space $\mathcal{V}$~\eqref{eq:defBR}\\
		$\ell(\R^d),\ell(\R^d\times\R^d)$ & the space $\R^d$ (or $\R^d\times\R^d$) endowed with vector (or matrix) $\ell^p$-norm $\|\cdot\|_p$\\
		$\Sigma^d_{r,R}(\Omega)$ & the set of matrix fields $\pmb{a}(\tilde{\pmb{x}})$ satisfying the uniform ellipticity assumption~\eqref{eq:defSigma_drR}\\
		$P_{Z_\Lambda}:Z\to Z_\Lambda$ & the metric projection from $Z$ to a compact convex subset $Z_\Lambda$~\eqref{eq:metricProjPK}\\
	\hline
	\end{tabular}
	\hspace*{-\textwidth}
\end{table}

\end{document}